\documentclass[10pt,a4paper,twoside,reqno]{amsart}
  
\usepackage{xcolor,soul,lipsum}
\usepackage{hyperref,url}
\usepackage{amssymb,latexsym}
\usepackage{amsmath,amsthm}
\usepackage[latin1]{inputenc}
\usepackage{enumerate}
\usepackage[all]{xy} \CompileMatrices
\usepackage{amscd}
\usepackage{comment}
\hypersetup{colorlinks=false,pdfborderstyle={/S/U/W 0}}
\usepackage{slashed} 
\usepackage{tikz}
\usetikzlibrary{matrix}
\usepackage{multirow}
\SelectTips{cm}{12} 
\theoremstyle{plain}
\usepackage{pict2e}
\newtheorem{theorem}{Theorem}[section]
\newtheorem{lemma}[theorem]{Lemma}
\newtheorem{corollary}[theorem]{Corollary}
\newtheorem{proposition}[theorem]{Proposition}

\theoremstyle{definition}
\newtheorem{definition}[theorem]{Definition}
\newtheorem{definition-theorem}[theorem]{Definition-Theorem}
\newtheorem{example}[theorem]{Example}
\theoremstyle{remark}
\newtheorem{remark}[theorem]{Remark}

\numberwithin{equation}{section} 

\usepackage{geometry}
\usepackage{enumitem}

\newgeometry{left=3cm,right=3cm,top=2.5cm,bottom=2.5cm}

\makeatletter
\DeclareRobustCommand{\loplus}{\mathbin{\mathpalette\dog@lsemi{+}}}
\DeclareRobustCommand{\lotimes}{\mathbin{\mathpalette\dog@lsemi{\times}}}
\DeclareRobustCommand{\roplus}{\mathbin{\mathpalette\dog@rsemi{+}}}
\DeclareRobustCommand{\rotimes}{\mathbin{\mathpalette\dog@rsemi{\times}}}

\newcommand{\dog@rsemi}[2]{\dog@semi{#1}{#2}{-90,90}}
\newcommand{\dog@lsemi}[2]{\dog@semi{#1}{#2}{270,90}}
\newcommand{\dog@semi}[3]{%
  \begingroup
  \sbox\z@{$\m@th#1#2$}%
  \setlength{\unitlength}{\dimexpr\ht\z@+\dp\z@\relax}%
  \makebox[\wd\z@]{\raisebox{-\dp\z@}{%
    \begin{picture}(1,1)
    \linethickness{\variable@rule{#1}}
    \roundcap
    \put(0.5,0.5){\makebox(0,0){\raisebox{\dp\z@}{$\m@th#1#2$}}}
    \put(0.5,0.5){\arc[#3]{0.5}}
    \end{picture}%
  }}%
  \endgroup
}
\newcommand{\variable@rule}[1]{%
  \fontdimen8  
  \ifx#1\displaystyle\textfont3\else
    \ifx#1\textstyle\textfont3\else
      \ifx#1\scriptstyle\scriptfont3\else
        \scriptscriptfont3\relax
  \fi\fi\fi
}
\makeatother

\newcommand{\tll}{\Theta_{ll}}
\newcommand{\tnn}{\Theta_{nn}}
\newcommand{\tln}{\Theta_{ln}}
\newcommand{\tun}{\Theta_{un}}
\newcommand{\tul}{\Theta_{ul}}
\newcommand{\tuu}{\Theta_{uu}}

\newcommand{\surj}{\to\kern-1.8ex\to}

\newcommand{\cF}{\mathcal{F}}
\newcommand{\cP}{\mathcal{P}}
\newcommand{\cG}{\mathcal{G}}

\newcommand{\cA}{\mathcal{A}}
\newcommand{\cK}{\mathcal{K}}
\newcommand{\cB}{\mathcal{B}}

\newcommand{\cC}{\mathcal{C}}
\newcommand{\Conf}{\mathrm{Conf}}
\newcommand{\Sol}{\mathrm{Sol}}

\newcommand{\fre}{\mathfrak{e}}
\newcommand{\frf}{\mathfrak{f}}
\newcommand{\G}{\mathrm{G}}
\newcommand{\cH}{\mathcal{H}}

\newcommand{\mS}{\mathrm{S}}
\newcommand{\cL}{\mathcal{L}}
\newcommand{\Spin}{\mathrm{Spin}}

\newcommand{\arxiv}[1]{{\tt
		\href{http://www.arXiv.org/abs/#1}{arXiv:#1}}}

\newcommand{\cI}{\mathcal{I}}
\newcommand{\dd}{\mathrm{d}}
\newcommand{\Cl}{\mathrm{Cl}}

\newcommand{\Gl}{\mathrm{Gl}}

\newcommand{\U}{\mathcal{U}}

\newcommand{\cY}{\mathcal{Y}}

\begin{document}

\title[Supersymmetric Kundt four manifolds and their spinorial evolution flows]{Supersymmetric Kundt four manifolds and their spinorial evolution flows}
 
\author[\'A. Murcia]{\'Angel Murcia}\address{Istituto Nazionale di Fisica Nucleare, Sezione di Padova, Repubblica Italiana} 
\email{angel.murcia@csic.es}

\author[C. S. Shahbazi]{C. S. Shahbazi} \address{Departamento de Matem\'aticas, UNED - Madrid, Reino de Espa\~na}
\email{cshahbazi@mat.uned.es} 
\address{Fachbereich Mathematik, Universit\"at Hamburg, Bundesrepublik Deutschland.}
\email{carlos.shahbazi@uni-hamburg.de}


\maketitle

\begin{abstract} 
We investigate the differential geometry and topology of four-dimensional Lorentzian manifolds $(M,g)$ equipped with a real Killing spinor $\varepsilon$, where $\varepsilon$ is defined as a section of a bundle of irreducible real Clifford modules satisfying the Killing spinor equation with non-zero real constant. Such triples $(M,g,\varepsilon)$ are precisely the supersymmetric configurations of minimal AdS four-dimensional supergravity and necessarily belong to the class Kundt of space-times, hence we refer to them as \emph{supersymmetric Kundt configurations}. We characterize a class of Lorentzian metrics on $\mathbb{R}^2\times X$, where $X$ is a two-dimensional oriented manifold, to which every supersymmetric Kundt configuration is locally isometric, proving that $X$ must be an elementary hyperbolic Riemann surface when equipped with the natural induced metric. This yields a class of space-times that vastly generalize the Siklos class of space-times describing gravitational waves in AdS$_4$. Furthermore, we study the Cauchy problem posed by a real Killing spinor and we prove that the corresponding evolution problem is equivalent to a system of differential flow equations, the \emph{real Killing spinorial flow equations}, for a family of functions and coframes on any Cauchy hypersurface $\Sigma\subset M$. Using this formulation, we prove that the evolution flow defined by a real Killing spinor preserves the Hamiltonian and momentum constraints of the Einstein equation with negative curvature and is therefore compatible with the latter. Moreover, we explicitly construct all left-invariant evolution flows defined by a Killing spinor on a simply connected three-dimensional Lie group, classifying along the way all solutions to the corresponding constraint equations, some of which also satisfy the constraint equations associated to the Einstein condition.  
\end{abstract}



\section{Introduction}
\label{sec:intro}



\subsection*{Motivation and context}


A smooth \emph{supersymmetric configuration} of a supergravity theory consists of a Lorentzian manifold $(M,g)$ possibly equipped with additional geometric structures such as a connection on a principal bundle, a curving on a gerbe, or a generalized metric on Courant algebroid, admitting a spinor parallel under a connection that depends on these additional geometric structures, which are in turn required to satisfy themselves a coupled spinorial differential system \cite{Ortin}. A supersymmetric configuration is called a \emph{supersymmetric solution} if in addition satisfies the equations of motion of the supergravity theory under consideration. Due to their outstanding physical and mathematical applications, the local theory of supersymmetric configurations and solutions has been extensively developed in the literature, see for instance \cite{Gran:2018ijr,Ortin} and their references and citations. Despite these efforts, the global differentiable topology and geometry of supersymmetric configurations remains poorly understood and largely inaccessible except in some notable cases \cite{Gibbons:2013tqa,Niehoff:2016gbi}. 

In \cite{Murcia:2020zig,Murcia:2021psf} we initiated a research program aimed at developing the differentiable topology and geometry of four-dimensional supersymmetric configurations by using the framework of spinorial polyforms \cite{Cortes:2019xmk}. This formalism allows to describe supersymmetric configurations as solutions to a differential system of equations for a polyform that belongs to a certain semi-algebraic body in the K\"ahler-Atiyah bundle of $(M,g)$ \cite{Lazaroiu:2012du}. In \cite{Murcia:2020zig,Murcia:2021psf} we considered the simplest type of supersymmetric configuration in four dimensions, namely Lorentz four-manifolds equipped with a parallel spinor, studying the differentiable topology of compact Cauchy hypersurfaces and the corresponding evolution problem, proving an initial data characterization for parallel spinors on Ricci-flat four-manifolds and solving it explicitly in the left-invariant case on a simply connected Lie group. In this note we continue this study with arguably the next case in difficulty, that is, the case of Lorentzian four-manifolds $(M,g)$ equipped with a real Killing spinor $\varepsilon$, which by definition is a solution of the following differential equation:
\begin{equation*}
\nabla^g_v\varepsilon = \frac{\lambda}{2} v\cdot\varepsilon\, , \quad \forall\,\, v\in \mathfrak{X}(M)\, , \quad \lambda \in \mathbb{R}^{\ast}\, ,
\end{equation*}

\noindent
where $\nabla^g$ is the lift of the Levi-Civita connection $\nabla^g$ to $\mathrm{S}_g$. The latter is a bundle of irreducible real Clifford modules over the bundle of Clifford algebras of $(M,g)$, and can be constructed as an associated bundle to a spin structure on $(M,g)$ \cite{LazaroiuShahbazi,Lazaroiu:2017zpl}. Lorentzian four-manifolds $(M,g)$ admitting such a real Killing spinor are precisely the supersymmetric configurations of AdS minimal four-dimensional supergravity \cite[\S 5]{Ortin} (see Remark \ref{rem:parallel} for the case $\lambda=0$). Their study turns out to be remarkably more complicated and rich than the parallel spinor case and was pioneered in \cite{GibbonsRuback}, where a particular class of four-dimensional space-times admitting real Killing spinors was studied from the physical point of view. In the mathematical literature, the local structure of complex Killing spinors in Lorentzian signature was considered in \cite{Bohle:2003abk,Leitner}, although the results of \cite{Bohle:2003abk} do not cover the type of real Killing spinor equation that we consider since in our case the pseudo-norm of both $\varepsilon$ and its Dirac current, denoted by $u\in\Omega^1(M)$ below, identically vanish on $M$ (c.f. \cite[Theorem 22]{Bohle:2003abk}). On the other hand, the study of non-degenerate submanifolds of pseudo-Riemannian manifolds equipped with a (complex) Killing spinor has been recently pioneered in \cite{Conti}, see also \cite{Morel} for the low-dimensional Riemannian case. The results of \cite{Conti} connect directly to the initial value problem for Killing spinors that we consider in Section \ref{sec:GloballyHyperbolic} and which is in fact formulated on a non-degenerate and complete hypersurface.


\subsection*{Main results}


We proceed to describe now the main results of this letter, which are contained in sections \ref{sec:Spinors4d} and \ref{sec:GloballyHyperbolic}. In Section \ref{sec:Spinors4d} we describe real Killing spinors $\varepsilon$ on Lorentzian four-manifolds $(M,g)$ in terms of parabolic pairs \cite{Murcia:2020zig,Murcia:2021psf} satisfying the prescribed differential system \eqref{eq:RKSspinorgeneralII}. This characterization immediately implies the existence of a nowhere vanishing luminous vector field whose integral curves define an affinely parametrized congruence of non-twisting null geodesics with vanishing expansion and shear, implying that Lorentzian four-manifolds $(M,g)$ admitting a real Killing spinor $\varepsilon$ are \emph{Kundt} \cite{Kundt,Stephani}. Hence, we call such triples $(M,g,\varepsilon)$ \emph{supersymmetric Kundt configurations}. Using this formalism  we consider an adapted class of Killing spinors on conformally Brinkmann metrics on $\mathbb{R}^2\times X$, where $X$ is a two-dimensional oriented manifold, to which every supersymmetric Kundt configuration is locally isomorphic and to which we will refer as \emph{standard} supersymmetric Kundt configurations. In Theorem \ref{thm:conformallyBrinkmann} we prove that the following characterization of standard supersymmetric Kundt configurations, which is in agreement with the results of \cite{BaumI,BaumII,Lewandowski:1991bx}, where the related notion of \emph{twistor spinor} was studied.

\begin{theorem}
A triple $(\mathbb{R}^2\times X,g, \varepsilon)$ is a standard supersymmetric Kundt configuration if and only if there exist a family $\left\{\omega_{x_u}\right\}_{x_u\in \mathbb{R}}$ of closed one-forms on $X$ in terms of which the metric $g$ reads:    
\begin{eqnarray}
\label{eq:thmintro1}
g = \cH_{x_u} \dd x_u\otimes \dd x_u +   e^{\cF_{x_u}} \dd x_u\odot ( \dd x_v + \beta_{x_u}) + \frac{1}{4\lambda^2} \dd_X \cF_{x_{u}}\otimes \dd_X\cF_{x_{u}} + e^{\cF_{x_u}} \omega_{x_u}\otimes \omega_{x_u}
\end{eqnarray}
	
\noindent
where $\left\{ \beta_{x_u}\right\}_{x_u\in \mathbb{R}}$ is a family of one-forms on $X$ satisfying the following equation:
\begin{equation}
\label{eq:thmintro2}
\dd_X \beta_{x_{u}} = \frac{e^{-\cF_{x_{u}}}}{4\lambda^2} \langle \partial_{x_u} \ast_{q_{x_u}}\dd_X \cF_{x_{u}} + \ast_{q_{x_u}}\partial_{x_u}\dd_X\cF_{x_{u}} , \dd_X \cF_{x_{u}} \rangle_{q_{x_u}} \nu_{q_{x_u}}\, .
\end{equation}
	
\noindent
In particular, for every $x_u\in\mathbb{R}$ the pair:
\begin{equation*}
(X,q_{x_u} = \frac{1}{4\lambda^2} \dd_X \cF_{x_{u}}\otimes \dd_X\cF_{x_{u}} + e^{\cF_{x_u}} \omega_{x_u}\otimes \omega_{x_u})
\end{equation*}
	
\noindent
is an elementary hyperbolic surface of scalar curvature $-2\lambda^2$, and therefore diffeomorphic to either $\mathbb{R}^2$ or $\mathbb{R}\times S^1$.
\end{theorem}

\noindent
The class of metrics ocurring in the previous theorem yields a vast generalization of the Siklos class of metrics \cite{Siklos} that represent an exact idealized wave moving through the AdS space-time \cite{Podolsky:1997ik,Podolsky:2002nn} and locally reduces to the latter when the families $\left\{\cF_{x_u}\right\}_{x_u\in \mathbb{R}}$ and $\left\{\omega_{x_u}\right\}_{x_u\in \mathbb{R}}$ do not depend on the coordinate $x_u\in\mathbb{R}$. Furthermore, the fact that $H^2(\mathbb{R}\times S^1,\mathbb{Z}) = 0$ implies that the differential equation \eqref{eq:thmintro2} that determines the family of one-forms $\left\{\beta_{x_u}\right\}_{x_u\in \mathbb{R}}$ always admits an infinite-dimensional vector space of solutions. The previous theorem implies the following local characterization of general supersymmetric Kundt configurations.

\begin{corollary}
Every four-dimensional space-time $(M,g)$ admitting a real Killing spinor is locally isometric to an open set of $\mathbb{R}^4$ equipped with the metric:
\begin{equation*}
g = \cH_{x_u} \dd x_u\otimes \dd x_u +   e^{\cF_{x_u}} \dd x_u\odot ( \dd x_v + \beta_{x_u}) + \frac{1}{4\lambda^2} \dd_X \cF_{x_{u}}\otimes \dd_X\cF_{x_{u}} + e^{\cF_{x_u}} \dd_X\cG_{x_u}\otimes \dd_X\cG_{x_u} 
\end{equation*} 
	
\noindent
for a family of pairs of functions $\left\{\cF_{x_u},\cG_{x_u}\right\}_{x_u\in \mathbb{R}}$ and of one-forms $\left\{\beta_{x_{u}}\right\}_{x_u}$ as prescribed in equation \eqref{eq:thmintro2}. 
\end{corollary}

\noindent
Note that the family of functions  $\left\{\cF_{x_u},\cG_{x_u}\right\}_{x_u\in \mathbb{R}}$ can be used to define local coordinates in which the metric $g$ simplies. This is however not possible globally in general, illustrating the global information contained in the previous theorem. Metric \eqref{eq:thmintro1} can be equivalently presented as follows:
\begin{equation*}
g = \frac{1}{\lambda^2 \cY_{x_u}^2} (\cK_{x_u}  \dd x_u\otimes \dd x_u + \dd x_u\odot (\dd x_v +   \lambda^2 \beta_{x_u}) + \dd_X \cY_{x_{u}}\otimes \dd_X\cY_{x_{u}} +  \omega_{x_u} \otimes \omega_{x_u})
\end{equation*}

\noindent
for families of functions $\left\{ \cY_{x_u} , \cK_{x_u}\right\}_{x_u\in \mathbb{R}}$. In this presentation it is apparent that $g$ can be interpreted as a deformation of the simply connected anti-de-Sitter space-time. Alternatively, the previous presentation of $g$  generalizes the ansatz considered in \cite{GibbonsRuback} to study Lorentzian metrics admitting Killing spinors since it gives through restriction the most general local form of a four-dimensional Lorentzian metric admitting Killing spinors.

In Section \ref{sec:GloballyHyperbolic} we consider globally hyperbolic Lorentzian four-manifolds equipped with a Killing spinor using the framework of parabolic pairs, which we exploit to conveniently reformulate the evolution problem and constraint equations posed by a real Killing spinor in terms of families of global coframes and functions. Using this formulation, we prove in Theorem \ref{thm:Killingspinorflow} that the evolution flow defined by a Killing spinor, to which we will refer as the \emph{Killing spinorial flow}, preserves both the Hamiltonian and momentum constraints of the negative-curvature Einstein equations. Furthermore, we explicitly solve the left-invariant Killing spinorial flow on a three-dimensional simply connected Lie group, classifying along the way all simply connected Lie groups that admit solutions to the corresponding constraint equations and verifying in addition  when they admit a compatible solution to the constraint equations of the Einstein condition. The classification of left-invariant initial data for the Killing spinorial flow is presented in Theorem \ref{theo:consks} in terms of the following table:

\

\begin{center}
\begin{tabular}{|  p{1cm}| p{9cm} | p{3.5cm} |}
\hline
$\mathrm{G}$ & \emph{Killing Cauchy pair} & \emph{Constrained Einstein}  \\ \hline
\multirow{2}*{$ \mathrm{E}(1,1)$} & \multirow{2}*{$\Theta=3\tnn e_u \otimes e_u+ 2 \lambda e_u \odot e_l-\tnn e_l \otimes e_l+\tnn e_n \otimes e_n$} & \multirow{2}*{\emph{Not allowed}}  \\ & &  \\ \hline \multirow{2}*{$ \tau_2 \oplus \mathbb{R}$} & \multirow{2}*{$\Theta=\Theta_{uu} e_u \otimes e_u+  \lambda e_u \odot e_l+\tnn e_n \otimes e_n$} & \multirow{2}*{$\tnn(\tnn-\tuu)=\lambda^2$} \\ & &  \\   \hline
\multirow{4}*{$ \tau_{3,\mu}$} & \multirow{3}*{$\Theta=\Theta_{uu} e_u \otimes e_u+ \tul  e_u \odot e_l+\tll e_l \otimes e_l+\tnn e_n \otimes e_n$} & \multirow{4}*{$\tul(2 \tul-3 \lambda )=0$}\\ & &   \\ & & \\ & $\tul \neq \lambda, 2\lambda\, , $ $  \lambda \tll= \left (\lambda-\tul \right )\tnn \, , $ $ \lambda \tuu=\left (\lambda+\tul\right )\tnn $ & \\\hline
\end{tabular}
\end{center}

\

\noindent
This table specifies the isomorphim type \cite{Milnor} of the simply connected Lie groups $\G$ that admit left-invariant initial data together with the corresponding second fundamental $\Theta$ expressed in terms of a left-invariant coframe. The term \emph{constrained Einstein} refers to the condition arising from imposing the constraint equations of Einstein equations. For more details the reader is referred to Theorem \ref{theo:consks}. The explicit form of all left-invariant flows is presented on a case by cases basis in Section \ref{sec:GloballyHyperbolic}, a result from which we can explicitly compute as a corollary the evolution of the Hamiltonian constraint $H_t$ in terms of its value at zero time $H_0$.

\begin{corollary}
Let $\left\{\beta_t,\fre^t\right\}_{t\in\cI}$ be a left-invariant real Killing spinor in $(M,g)$. 
\begin{itemize}[leftmargin=*]
\item If $\G=\tau_2$, then $H_t=\frac{4 \lambda^2 H_0}{\tuu^2+4 \lambda^2}  \sec^2(\arctan \left ( \frac{\tuu}{2\lambda}\right)+2\lambda \cB_t)$.
		
\item If $\G=\mathrm{E}(1,1)$, then $H_t=-4 \lambda^2 \sec^2 (\arctan\left (\frac{\tnn}{\lambda} \right) +3\lambda \cB_t)$.
		
\item If $\G=\tau_{3,\mu}$, then $H_t=\frac{\lambda^2 H_0}{\lambda^2+\tnn^2} \sec^2(\arctan \left ( \frac{\tnn}{\lambda} \right ) +(\tul+\lambda) \cB_t)$.
\end{itemize}
	
\noindent
where $\cB_t=\int_0^t \beta_\tau \dd \tau$ and $H_0$ is the Hamiltonian constraint at time $t=0$. 
\end{corollary}

\noindent
Clearly, when $G=\tau_2$ or $G=\tau_{3,\mu}$ the Hamiltonian constraint is zero at time $t$ if and only if it is zero at $t=0$. On the other hand, if $\G=\mathrm{E}(1,1)$ then the Hamiltonian constraint can never vanish, in agreement with the fact that in this case the initial data cannot satisfy the constraint equations of Einstein equations. 


\subsection*{Future directions of research}


The present article motivates several research lines in the area of four-dimensional Kundt space-times via the notion of supersymmetric Kundt configuration. Indeed, the class of four-dimensional Kundt space-times is vast \cite{Stephani} and only recently a fully global and differential-geometric characterization of these space-times, together with a partial classification result in the three-dimensional case, was given in \cite{Zeghib}, see also \cite{ColeyPapa}. The notion of \emph{supersymmetric} Kundt configuration determines a manageable subclass of Kundt space-times that can be accessed through spinorial geometry techniques and that admits a natural interpretation within the framework of four-dimensional supergravity. The first problem that we plan to address in the future is the study of the Einstein condition, which gives rise to the notion of supersymmetric Kundt \emph{solution}. We expect this condition to become a spectral problem on a hyperbolic Riemann surface for families of functions parametrized by $x_u\in \mathbb{R}$, potentially creating an interesting link between supersymmetric Kundt solutions in four Lorentzian dimensions and the theory of automorphic forms on a hyperbolic Riemann surface. Other basic properties of supersymmetric Kundt configurations or solutions seem to remain unexplored, such as global hyperbolicity, existence of horizons or geodesic completeness. It would be interesting for instance to obtain asymptotic growth conditions on $\left\{\cH_{x_{u}}\right\}_{x_u\in \mathbb{R}}$ that guarantee geodesic completeness, in the spirit of \cite{Candela:2002rr,Flores:2004dr} for Brinkmann space-times. Furthermore, it would be interesting to study the different types of conformal/ideal/causal boundaries that supersymmetric Kundt solutions admit. This is specially important because of the potential applications of this class of space-times, which have negative Ricci curvature, to the AdS/CFT duality.


\subsection*{Acknowledgements}


We would like to thank Bernardo Araneda, Diego Conti, Antonio F. Costa and Romeo Segnan Dalmasso for useful discussions, Calin Lazaroiu for reading a preliminary version of the draft and proposing several improvements and Patrick Meessen for various insightful remarks and pointers. The work of \'A.M. was  funded by the Spanish FPU Grant No. FPU17/04964 and by the Istituto Nazionale di Fisica Nucleare (INFN), through the INFN Call No. 23590. \'A.M. also received additional support from the MCIU/AEI/FEDER UE grant PID2021-125700NB-C21. CSS's research was funded by the Mar\'ia Zambrano Excellency Program 468806966 of the Kingdom of Spain.


\section{Four-dimensional supersymmetric Kundt configurations}
\label{sec:Spinors4d}


In this section we present the theory of real Killing spinors on four-dimensional Lorentzian manifolds. The starting point is one of the main results of \cite{Cortes:2019xmk}, according to which real Killing spinors can be reformulated in terms of a specific type of distribution satisfying a prescribed system of partial differential equations.


\subsection{General theory}
\label{subsec:GeneralSpinors4d}


Let $(M,g)$ be a \emph{space-time}, i.e., a connected, oriented and time oriented Lorentzian four-manifold $M$ endowed  with a Lorentzian metric $g$. Assume $(M,g)$ is also equipped with a bundle of irreducible real spinors $\mathrm{S}_g$. The existence of $\mathrm{S}_g$ is equivalent \cite{LazaroiuShahbazi,Lazaroiu:2017zpl,LazaroiuSha} to the existence of a spin structure $Q_g$, in which case $\mS_g$ can be regarded as the vector bundle associated to $Q_g$ through the tautological representation induced by the canonical embedding $\mathrm{Spin}_{0}(3,1)\subset \Cl(3,1)$, where $\mathrm{Spin}_{0}(3,1)$ stands for the component connected to the identity of the spin group in signature $(3,1) = -+++$ and $\Cl(3,1)$ represents the real Clifford algebra in the aforementioned signature. Consequently, one may always equip $(M,g)$ with a fixed spin structure $Q_g$. Then, the Levi-Civita connection $\nabla^g$ on $(M,g)$ induces naturally on $\mS_g$ the \emph{spinorial} Levi-Civita connection, which we denote for the sake of simplicity by the same symbol. 

\begin{definition}
A \emph{spinor} $\varepsilon$ on $(M,g,\mS_g)$, or a spinor $\varepsilon$ on $M$, is a smooth section $\varepsilon\in \Gamma(\mS_g)$ of $\mS_g$. It is a \emph{real Killing spinor} (or just Killing spinor) if:
\begin{equation*}
\nabla^g_v\varepsilon = \frac{\lambda}{2} v\cdot\varepsilon\, , \qquad \forall\,\, v\in \mathfrak{X}(M)\, ,
\end{equation*}

\noindent
where \emph{dot} stands for Clifford multiplication and $\lambda$ is a non-zero real constant, called the \emph{Killing constant} of $\varepsilon$. 
\end{definition}

\noindent
A triple $(M,g,\varepsilon)$ consisting of a Lorentzian four-manifold $(M,g)$ and a Killing spinor $\varepsilon$ on $(M,g)$ for some choice of spinor bundle $\mS_g$ will be called a \emph{supersymmetric Kundt configuration} since such $(M,g)$ is a supersymmetric configuration of minimal four-dimensional supergravity and, as explained in Subsection \ref{subsec:opticalinvariants}, is in addition necessarily Kundt.

\begin{remark}
\label{rem:parallel}
The case $\lambda = 0$ leads to the notion of \emph{parallel spinor}, which has been extensively studied in the literature, see \cite{BaumLeistnerLischewski,Brinkmann,EhlersKundt,LeistnerLischewski,Lischewski:2015cya,Murcia:2020zig,Tod} and references therein.
\end{remark}

\noindent
Let $u\in \Omega^1(M)$ be a light-like one-form. Given $l_1 , l_2 \in \Omega^1(M)$ , we state them to be equivalent,  $l_1 \sim_u l_2$, if and only if $l_1 = l_2 + f u$ for a function $f\in C^{\infty}(M)$. We denote by:
\begin{equation*}
\Omega^1_{u}(M) := \frac{\Omega^1(M)}{\sim_u}\, ,
\end{equation*}

\noindent
the vector space of equivalence classes introduced by $\sim_u$.  

\begin{definition}
A parabolic pair $(u,[l])$ on $(M,g)$ is conformed by a nowhere-vanishing null one-form $u\in \Omega^1(M)$ and an equivalence class of one-forms $[l]\in \Omega^1_u(M)$ such that:
\begin{equation*}
g(l,u) =  0\, , \qquad g(l,l) = 1\, ,
\end{equation*}

\noindent
for any, and hence, for all, representatives $l\in [l]$.
\end{definition}

\noindent
From \cite[Theorem 5.3]{Cortes:2019xmk}, we may express real Killing spinors on $(M,g)$ in terms of a prescribed system of partial differential equations for a pair of one-forms.

\begin{proposition}
\label{prop:Killingspinorgeneral}
An oriented and time-oriented Lorentzian four-manifold $(M,g)$ admits a real Killing spinor $\varepsilon\in \Gamma(\mS_g)$ if and only if there exists a parabolic pair $(u,[l])$ on $(M,g)$ such that:
\begin{equation}
\label{eq:RKSspinorgeneral}
\nabla^g u = \lambda\, u\wedge l\, , \qquad \nabla^g l = \kappa\otimes u + \lambda (l\otimes l - g)\, ,
\end{equation}

\noindent
for some representative $l\in [l]$ and one-form $\kappa \in \Omega^1(M)$.  
\end{proposition}

\noindent
More concretely, \cite{Cortes:2019xmk} proves that a nowhere-vanishing spinor $\varepsilon \in \Gamma(\mathrm{S}_g)$ on $(M,g)$ defines a unique distribution of co-oriented parabolic two-planes in $M$. This fixes in turn both $u$ and $[l]$ uniquely. Conversely, any such distribution specifies a unique nowhere-vanishing spinor on $(M,g)$ (with respect to a given spin structure on $(M,g)$) up to a global sign. In fact, \cite[Theorem 4.26]{Cortes:2019xmk} states the equivalence between certain differential equations for spinors $\varepsilon$ and fixed systems of differential equations for $(u,[l])$, of which equations \eqref{eq:RKSspinorgeneral} are a particular case. 

\noindent
The first equation in \eqref{eq:RKSspinorgeneral} is explicitly given by:
\begin{equation*}
\nabla^g_v u = \lambda (u(v)\, l - l(v)\,u)\, ,
\end{equation*}

\noindent
for every $v\in\mathfrak{X}(M)$. Since $l$ is necessarily nowhere vanishing, this equation does \emph{not} reduce to the standard equation of a recurrent light-like vector field \cite{Galaev:2009ie,Gibbons:2007zu} and seems to be new in the mathematical literature. Since every Killing spinor on $(M,g)$ implies the existence of a luminous one-form $u\in \Omega^1(M)$ satisfying the first equation in \eqref{eq:RKSspinorgeneral}, we obtain a reduction of the orthonormal frame bundle of $(M,g)$ to a \emph{Bargmannian} or \emph{null} structure \cite{Figueroa-OFarrill:2020gpr,Papadopoulos:2018mma}, see also \cite{Fino:2020uuz} for the very related notion of \emph{optical geometry}. In our case, and using the notation of \cite{Figueroa-OFarrill:2020gpr}, Proposition \ref{prop:Killingspinorgeneral} immediately implies that every real Killing spinor induces a Bargmannian structure of type $\cB_{16}$. The existence of the class of one-forms $[l]$ in addition to the luminous one-form $u\in \Omega^1(M)$ as part of a parabolic pair implies a further frame-bundle reduction with torsion to the subgroup of null rotations of the Lorentz subgroup in $\Gl(4,\mathbb{R})$ that is determined as the stabilizer of $(u,[l])$. This subgroup is isomorphic to $\mathbb{R}^2$, which, since it is connected and simply connected, is in turn isomorphic to the stabilizer of a real an irreducible spinor in $\Spin(3,1)$. 

\begin{remark}
The local structure of \emph{imaginary} complex Killing spinors, of which the \emph{real} Killing spinors we consider are a particular case, has been studied in \cite{Leitner} using methods different to those introduced in this article. By the results of Op. Cit. it follows that Lorentzian four-manifolds admitting real Killing spinors are locally conformally Brinkmann, a fact that is reflected in our definition of \emph{standard} Brinkmann space-time that we introduce in subsection \ref{subsec:adapted}. In contrast, the global geometric and topological structure of Lorentzian manifolds admitting real Killing spinors does not appear to have been studied in the literature, at least systematically. Proposition \ref{prop:Killingspinorgeneral} is particularly adequate to initiate such investigation.
\end{remark}

\begin{lemma}
\label{lemma:dependencerep}
Let $(u,[l])$ be a parabolic pair on $(M,g)$ and assume that $(u,l)$, where $l\in [l]$, satisfies equations \eqref{eq:RKSspinorgeneral} with respect to $\kappa\in \Omega^1(M)$. Then, for any other representative $\hat{l} \in [l]$, the pair $(u,\hat{l})$ also satisfies equations \eqref{eq:RKSspinorgeneral}  with respect to:
\begin{equation}
\hat{\kappa} = \kappa + \dd f - 2 \lambda f \hat{l} -\lambda f^2 u\, .
\end{equation}

\noindent
where $f\in C^{\infty}(M)$ is a function such that $\hat{l} = l + f u$. 
\end{lemma}

\noindent
A parabolic pair $(u,[l])$ is \emph{Killing} if it corresponds to a Killing spinor, i.e., if it satisfies equations \eqref{eq:RKSspinorgeneral} for a representative $l\in [l]$. Lemma \ref{lemma:dependencerep} shows that a parabolic pair $(u,[l])$ being Killing does not depend on the representative chosen.  


\subsection{Optical invariants of supersymmetric Kundt configurations} 
\label{subsec:opticalinvariants}


By projecting equations \eqref{eq:RKSspinorgeneral} to their symmetric and skew-symmetric components we obtain the following equivalent system of differential equations, which is arguably more convenient to do explicit computations:
\begin{equation}
\label{eq:RKSspinorgeneralII}
\cL_{u^{\sharp_g}}g= 0\, , \quad \dd u = 2 \lambda\, u\wedge l\, , \quad \dd l = \kappa\wedge u\, , \quad \cL_{l^{\sharp_g}} g=\kappa \odot u+2 \lambda( l\otimes l-g)
\end{equation}

\noindent
Both the one-form $u$ and its metric dual $u^{\sharp_g}$ are usually referred to in the literature as the \emph{Dirac current} of $\varepsilon$. By the first equation in \eqref{eq:RKSspinorgeneral}, the metric dual $u^{\sharp_g}\in \mathfrak{X}(M)$ of $u$ is a luminous Killing vector field with geodesic integral curves. Therefore, the Dirac current $u^{\sharp_g}$ of a real Killing spinor $\varepsilon$ generates a \emph{congruence} of null geodesics, namely a dimension one foliation of $M$ by affinely parametrized oriented geodesics tangent to $u^{\sharp_g}$. This congruence of null geodesics admits a number of invariants which have been extensively studied in the mathematical general relativity literature under the name of \emph{optical invariants}. In the present case, we have three independent optical invariants, namely the \emph{expansion} $\theta$, the \emph{twist} $\omega^2$ and the \emph{shear} $\sigma^2$, respectively defined as follows \cite{Chandra}:
\begin{equation*}
\theta := \frac{1}{2} \mathrm{div}_{g} u\, , \qquad \omega^2 := \frac{1}{4} \vert \mathrm{d} u\vert_g^2\, , \qquad \sigma^2 := \frac{1}{8} \vert \cL_{u^{\sharp_g}}g\vert_g^2 - \theta^2\, ,
\end{equation*}

\noindent
where $\mathrm{div}_g\colon \Omega^1(M)\to \cC^{\infty}(M)$ is the divergence operator associated to $g$. A quick computation shows that if $(u,[l])$ is a Killing parabolic pair, then all the optical invariants vanish, that is, $\theta = \omega^2 = \sigma^2 = 0$. Therefore, Lorentzian four-manifolds admitting a Killing spinor are canonically equipped with a non-expanding, non-twisting and non-shear geodesic null congruence and thus belong to the Kundt class of four-dimensional space-times \cite{Kundt,Stephani} justifying the term \emph{supersymmetric Kundt configurations}. The fact that Kundt space-times occur as supersymmetric solutions of supergravity theories was noticed in \cite{Brannlund:2008zf}, where a different type of Kundt space-times to those considered in this letter was studied in the context of supersymmetric solutions of supergravity. Indeed, the Kundt space-times considered in Op. Cit. are characterized by admitting a recurrent light-like vector field preserved by the Levi-Civita connection. As explained above this is not the case for the supersymmetric Kundt configurations that we consider.

Alternatively, equations \eqref{eq:RKSspinorgeneral} immediately imply condition (K2) in \cite{Zeghib} with the one-form $\alpha$ occurring in Op. Cit. equal to $\lambda l$. Therefore, we can think of supersymmetric Kundt configurations as Kundt space-times for which the one-form $\alpha$ appearing in equation (2) of \cite{Zeghib} is required to be of unit norm and to satisfy the second differential equation in \eqref{eq:RKSspinorgeneral}. Since the existence of a Killing spinor on $(M,g)$ is naturally compatible with the Einstein condition for $g$, we may expect the class of Lorentzian manifolds admitting Killing spinors to be related with the many classes of Einstein Kundt space-times that have been extensively studied the mathematical general relativity community, see for instance \cite{Bicak:1999ha,Bicak:1999hb,Podolsky:1997ik,Podolsky:1999sw,Podolsky:2002nn,Podolsky:2002sy,Podolsky:2008eg} and their references and citations.


\subsection{Standard supersymmetric Kundt configurations} 
\label{subsec:adapted}


In this subsection we consider a class of supersymmetric Kundt configurations that includes as a particular case the local isomorphism type of every supersymmetric Kundt configuration. 

\begin{definition}
A four-dimensional space-time $(M,g)$ is \emph{standard conformally Brinkmann} if $(M,g)$ has the following isometry type: 
\begin{equation}
\label{eq:isostanads}
(M,g) = (\mathbb{R}^2\times X, \cH_{x_u} \dd x_u\otimes \dd x_u + e^{\cF_{x_u}} \dd x_u\odot \dd x_v + \dd x_u \odot \alpha_{x_u} + q_{x_u})\, ,
\end{equation}
	
\noindent
where $(x_u,x_v)$ are the Cartesian coordinates of $\mathbb{R}^2$, and:
\begin{equation*}
\left\{ \cH_{x_{u}} ,\cF_{x_{u}} \right\}_{x_{u}\in \mathbb{R}}\, , \quad \left\{\alpha_{x_u} \right\}_{x_{u}\in \mathbb{R}}\, , \quad \left\{ q_{x_u}\right\}_{x_{u}\in \mathbb{R}}\, ,
\end{equation*}
	
\noindent
respectively denote a family of pairs of functions, a family of one-forms and a family of complete Riemannian metrics on $X$ parametrized by $x_{u}\in \mathbb{R}$.
\end{definition} 

\begin{definition}
A spinor $\varepsilon$ on a standard conformally Brinkmann space-time $(M,g)$ is \emph{adapted} if the parabolic pair $(u,[l])$ associated to $\varepsilon$ satisfies $u^{\sharp} = \partial_{x_v}$, in which case we will refer to $(u,[l])$ as an \emph{adapted parabolic pair} on $(M,g)$. 
\end{definition}

\noindent
By \cite[Lemma 5.10]{Cortes:2019xmk} every supersymmetric Kundt configuration is locally isomorphic to a standard conformally Brinkmann space-time equipped with an adapted Killing spinor. 

\begin{definition}
A supersymmetric Kundt configuration $(M,g,\varepsilon)$ is \emph{standard} if $(M,g)$ is a standard conformally Brinkmann space-time and the Killing spinor $\varepsilon$, or equivalently its associated parabolic pair $(u,[l])$, are adapted.	
\end{definition}

\noindent
We will characterize standard supersymmetric Kundt configurations $(M,g,\varepsilon)$ in the following. Let $\varepsilon$ be an adapted spinor with associated Killing pair $(u,[l])$ on a standard conformally Brinkmann space-time $(M,g)$. Since $u^{\sharp} = \partial_{x_v}$ by definition, the first equation in \eqref{eq:RKSspinorgeneralII} is automatically satisfied. Using that  $u = e^{\cF_{x_u}} \dd x_u$, the second equation in \eqref{eq:RKSspinorgeneralII} reduces to:
\begin{equation}
(\lambda\, l + \frac{1}{2} \dd_X \cF_{x_u})\wedge u = 0 \, ,
\end{equation}

\noindent
for any representative $l\in [l]$. The general solution to this equation reads:
\begin{equation*}
l=\ell_{x_u}+\sigma_l u\, , \qquad \ell_{x_u}=-\frac{1}{2\lambda} \dd_X \cF_{x_u}\, ,
\end{equation*}

\noindent
for a function $\sigma_l\in C^\infty(M)$ and a family of functions $\{\cF_{x_u}\}_{x_u \in \mathbb{R}}$ in $X$. By Lemma \ref{lemma:dependencerep}, we assume, without loss of generality, the following expression for $l$:
\begin{equation}
\label{eq:expressionl}
l = \ell_{x_u} = -\frac{1}{2\lambda} \dd_X \cF_{x_u}\, .
\end{equation}

\noindent
With this choice, the third equation in \eqref{eq:RKSspinorgeneralII} reduces to:
\begin{equation*}
\dd l=\dd x_u \wedge \partial_{x_u} \ell_{x_u} + \dd_X \ell_{x_u} = \dd x_u \wedge \partial_{x_u} \ell_{x_u} = e^{\cF_{x_{u}}}\kappa \wedge \dd x_u\,.
\end{equation*}

\noindent
where we have used that $\dd_X \ell_{x_u}=0$ by equation \eqref{eq:expressionl}. The previous equation  is solved by isolating $\kappa$ as follows:
\begin{equation}
\label{eq:algunaskapicas}
\kappa = -e^{-\cF_{x_{u}}}\partial_{x_u} \ell_{x_u} + \sigma_\kappa \dd x_u\, ,
\end{equation}

\noindent
for a function $\sigma_\kappa\in C^\infty(M)$. Recalling that we must have $\vert l \vert_g^2=1$, the previous discussion implies the following result.

\begin{proposition}
Let $(u,[l])$ be an adapted Killing parabolic pair on a standard conformally Brinkmann space-time. Then, there exists a representative $l\in[l]$ such that $(u,l)$ is of the form:
\begin{equation*}
u = e^{\cF_{x_{u}}} \dd x_u\, , \qquad l = -\frac{1}{2\lambda} \dd_X \cF_{x_u}
\end{equation*}

\noindent
where the family of functions $\left\{\cF_{x_{u}}\right\}_{x_u\in \mathbb{R}}$ satisfies the differential equation:
\begin{equation*}
\vert \dd_X \cF_{x_{u}}\vert_{q_{x_u}}^2 = 4 \lambda^2\, .
\end{equation*}

\noindent
Here $\vert \cdot \vert_{q_{x_u}}^2$ denotes the norm defined by $\left\{q_{x_{u}}\right\}_{x_u\in \mathbb{R}}$ on $X$.
\end{proposition}

\noindent
It only remains to solve the fourth equation in \eqref{eq:RKSspinorgeneralII}. For this, we need first the following lemma.

\begin{lemma}
\label{lemma:Lieg}
The Lie derivative of $g$ with respect to  $l^{\sharp_g} = \ell_{x_u}^{\sharp_g}$ is given by:
\begin{eqnarray*}
& \cL_{l^{\sharp_g}} g = ( \dd_X \cH_{x_u}(\ell_{x_u}^{\sharp_q})+2\alpha_{x_u} ( \partial_{x_u} \ell_{x_u}^{\sharp_q}) + 2e^{\cF_{x_{u}}}\partial_{x_u} \frf_{x_u} )  \dd x_u \otimes \dd x_u+ \dd_X e^{\cF_{x_{u}}} (\ell_{x_u}^{\sharp_q}) \dd x_u \odot \dd x_v\\ & + e^{\cF_{x_{u}}} \dd_X \frf_{x_u} \odot \dd x_u+ \cL_{\ell_{x_u}^{\sharp_q}}^X \alpha_{x_u} \odot \dd x_u +\cL_{\ell_{x_u}^{\sharp_q}}^X q_{x_u}+\partial_{x_u} \ell_{x_u}  \odot \dd x_u- (\partial_{x_u} q_{x_u}) (\ell_{x_u}^{\sharp_q}) \odot \dd x_u\, , 
\end{eqnarray*}

\noindent
where the symbol $\cL^{X}$ denotes the Lie derivative operator on $X$ and $\left\{ \frf_{x_u}\right\}_{x_u\in \mathbb{R}}$ is the family of functions on $X$ determined by:
\begin{equation*}
\frf_{x_u} = \frac{e^{-\cF_{x_{u}}}}{2\lambda} \dd_X \cF_{x_u}(\alpha^{\sharp_{q}}_{x_u})\, ,
\end{equation*}

\noindent
for every $x_u\in \mathbb{R}$. The symbol $\sharp_q$ denotes the musical isomorphism on $X$ determined by $\left\{ q_{x_u}\right\}_{x_u\in \mathbb{R}}$.
\end{lemma}

\begin{proof}
We first compute the metric dual of $\ell_{x_u}$ with respect to $g$, which is given by:
\begin{equation}
\label{eq:defil}
l^{\sharp_g}=\ell_{x_u}^{\sharp_{q}} - e^{-\cF_{x_{u}}} \alpha_{x_u}(\ell_{x_u}^{\sharp_{q}}) \partial_{x_v}\, ,
\end{equation} 

\noindent 
Using the previous expression for $l^{\sharp_g}$ compute:
\begin{eqnarray*}
& \cL_{l^{\sharp_g}} \dd x_u = 0\, , \qquad \cL_{l^{\sharp_g}} \dd x_v=\dd_X \frf_{x_u}+\partial_{x_u} \frf_{x_u} \dd x_u \, , \qquad \cL_{l^{\sharp_g}} \alpha_{x_u}=\cL_{\ell_{x_u}^{\sharp_q}}^X \alpha_{x_u}+\alpha_{x_u}\left ( \partial_{x_u} \ell_{x_u}^{\sharp_q} \right) \dd x_u \\ 
&\cL_{l^{\sharp_g}} q_{x_u}  =\cL_{\ell_{x_u}^{\sharp_q}}^X q_{x_u}+q_{x_u}(\partial_{x_u} \ell^{\sharp_{q}}_{x_u})  \odot \dd x_u =\cL_{\ell_{x_u}^{\sharp_q}}^X q_{x_u}+\partial_{x_u} \ell_{x_u}  \odot \dd x_u- (\partial_{x_u} q_{x_u}) (\ell_{x_u}^{\sharp_q}) \odot \dd x_u
\end{eqnarray*}

\noindent
from which the Lie derivative of $g$ as given in equation \eqref{eq:isostanads} follows directly.
\end{proof}

\begin{lemma}
\label{lemma:iffRKS}
A standard conformally Brinkmann space-time admits an adapted Killing spinor if and only if the tuple $\left\{\cF_{x_{u}}, \alpha_{x_u}, q_{x_u}\right\}_{x_u\in\mathbb{R}}$ satisfies the following differential system on $X$:
\begin{eqnarray}
& \vert \dd_X \cF_{x_{u}}\vert_{q_{x_u}}^2 = 4 \lambda^2\, , \quad \nabla^{q_{x_u}} \dd_X \cF_{x_{u}} + \frac{1}{2}  \dd_X \cF_{x_{u}}\otimes \dd_X \cF_{x_{u}} = 2 \lambda^2 q_{x_u}\label{eq:constraint2dX}\\ 
& 4\lambda^2  \dd_X \alpha_{x_u} = \langle \partial_{x_u} \ast_{q_{x_u}}\dd_X \cF_{x_{u}} + \ast_{q_{x_u}}\partial_{x_u}\dd_X\cF_{x_{u}} - 4\lambda^2 \ast_{q_{x_u}} \alpha_{x_u} , \dd_X \cF_{x_{u}} \rangle_{q_{x_u}} \nu_{q_{x_u}} \label{eq:evolution2dX}
\end{eqnarray}

\noindent
where $\langle\cdot,\cdot\rangle_{q_{x_u}}$ denotes the norm defined by $q_{x_u}$ and $\nu_{q_{x_u}}$ denotes the Riemannian volume form of $(X,q_{x_u})$.
\end{lemma} 



\begin{proof}
By the previous discussion we only need to consider the fourth equation in \eqref{eq:RKSspinorgeneralII} evaluated on a parabolic pair $(u,[l])$ of the form:
\begin{equation}
u = e^{\cF_{x_{u}}} \dd x_u \, , \qquad l = \ell_{x_u}=  -\frac{1}{2\lambda} \dd_X \cF_{x_u}\, ,
\end{equation}

\noindent
and with respect to a one-form $\kappa \in \Omega^1(\mathbb{R}^2\times X)$ given by:
\begin{equation*}
\kappa = -e^{-\cF_{x_{u}}}  \partial_{x_u} \ell_{x_u} + \sigma_\kappa \dd x_u\, ,
\end{equation*}

\noindent
where $\left\{\cF_{x_{u}}\right\}_{x_u\in\mathbb{R}}$ satisfies $\vert \dd_X \cF_{x_{u}}\vert_{q_{x_u}}^2 = 4 \lambda^2$. Plugging Lemma \ref{lemma:Lieg} into the fourth equation in \eqref{eq:RKSspinorgeneralII} and isolating we obtain the following system of equivalent equations: 
\begin{eqnarray}
\label{eq:kapicau} & 2\sigma_\kappa e^{\cF_{x_{u}}}  = 2\lambda \cH_{x_u} + \dd_X \cH_{x_{u}}(\ell_{x_u}^{\sharp_q})+ 2 e^{\cF_{x_{u}}}  \partial_{x_u} \frf_{x_u} +  2\alpha_{x_u} ( \partial_{x_u} \ell_{x_u}^{\sharp_q}) \\ 
\label{eq:lieq} & \dd_X \cF_{x_u} ( \ell_{x_u}^{\sharp_q}) = - 2 \lambda \, , \quad \cL_{\ell_{x_u}^{\sharp_q}}^X q_{x_u} =2 \lambda( \ell_{x_u} \otimes \ell_{x_u}- q_{x_u})\\
\label{eq:killings3}& e^{\cF_{x_{u}}}  \dd_X \frf_{x_u}  = (\partial_{x_u} q_{x_u})(\ell_{x_u}^{\sharp_q}) - 2 \lambda \alpha_{x_u} - \cL^X_{\ell_{x_u}^{\sharp_q}}\alpha_{x_u} - 2\partial_{x_u} \ell_{x_u} 
\end{eqnarray}

\noindent
Equation \eqref{eq:kapicau} is solved by isolating $\sigma_\kappa$, which determines the latter unambiguously. The first equation in \eqref{eq:lieq} follows from the first equation in \eqref{eq:constraint2dX} together with equation \eqref{eq:expressionl} whereas the second equation in \eqref{eq:lieq} is equivalent to the second equation in \eqref{eq:constraint2dX} upon use of equation \eqref{eq:expressionl}. On the other hand, equation \eqref{eq:killings3} can be shown to be equivalent to:
\begin{equation}
2 \lambda \alpha_{x_u}+ \iota_{\ell_{x_u}^{\sharp_q}} \dd_X \alpha_{x_u}-\frac{1}{2\lambda} \dd_X \cF_{x_u} (\alpha_{x_u}^{\sharp_q}) \dd_X \cF_{x_u} + 2 \partial_{x_u} \ell_{x_u}- (\partial_{x_u} q_{x_u})(\ell_{x_u}^{\sharp_q})=0\, .
\end{equation}

\noindent
Projecting the previous equation along $\ell_{x_u}^{\sharp_q}$ we obtain an identity. On the other hand, projecting along $(\ast_{q_{x_u}}\ell_{x_u})^{\sharp_q}$ we obtain, after some manipulations, equation  \eqref{eq:evolution2dX} and hence we conclude.
\end{proof}

\begin{remark}
Recall that equations \eqref{eq:constraint2dX} and \eqref{eq:evolution2dX} do not involve $\left\{ \cH_{x_{u}}\right\}_{x_u\in\mathbb{R}}$, which hence can be chosen at will while preserving the existence of an adapted Killing spinor.
\end{remark}
 
\begin{proposition}
\label{prop:standardqxu}
Let $(M,g,\varepsilon)$ be a standard supersymmetric Kundt configuration. Then, there exists a diffeomorphism identifying either $X= \mathbb{R}^2$ or $X= \mathbb{R}\times S^1$ and a smooth family of closed one-forms $\left\{\omega_{x_u}\right\}_{x_u\in\mathbb{R}}$ on $X$ such that:
\begin{equation}
\label{eq:standardqxu}
q_{x_u} = \frac{1}{4\lambda^2} \dd_X \cF_{x_{u}}\otimes \dd_X\cF_{x_{u}} + e^{\cF_{x_u}} \omega_{x_u}\otimes \omega_{x_u}\, .
\end{equation}

\noindent
In particular, for every $x_u\in \mathbb{R}$ the Riemann surface $(X , q_{x_u})$ is a hyperbolic Riemann surface of constant scalar curvature $\mathrm{Scal}^{q_{x_u}} = -2\lambda^2$.
\end{proposition}

\begin{proof}
For any fixed $x_u\in \mathbb{R}$, the function $\cF_{x_{u}}$ on $(X , q_{x_u})$ has unit-norm gradient. Hence, since $(X , q_{x_u})$ is complete, $\cF_{x_{u}}(X) = \mathbb{R}$ and we have a diffeomorphism:
\begin{equation*}
	X = \mathbb{R}\times \cF_{x_{u}}^{-1}(0)\, .
\end{equation*}

\noindent
Since $X$ is assumed to be connected, either $\cF_{x_{u}}^{-1}(0)$ is diffeomorphic to $\mathbb{R}$ or $S^1$. On the other hand, the family of one-forms $\left\{ \ell_{x_u}\right\}_{x_u\in \mathbb{R}}$ has unit norm with respect to $\left\{ q_{x_u}\right\}_{x_u\in \mathbb{R}}$. Therefore, defining:
\begin{equation*}
n_{x_u} = \ast_{q_{x_u}} \ell_{x_u}\, , \qquad x_u\in \mathbb{R}\, ,
\end{equation*}

\noindent
we obtain a family $\left\{ \ell_{x_u},n_{x_u}\right\}_{x_u\in \mathbb{R}}$ of orthonormal coframes on $X$. The second equation in \eqref{eq:constraint2dX} implies the following equations for $\left\{ \ell_{x_u},n_{x_u}\right\}_{x_u\in \mathbb{R}}$:
\begin{equation*}
\nabla^{q_{x_u}}\ell_{x_u} = - \lambda n_{x_u}\otimes n_{x_u}\, , \qquad \nabla^{q_{x_u}} n_{x_u} = \lambda n_{x_u}\otimes \ell_{x_u}\, ,
\end{equation*}

\noindent
which in turn implies, since $\left\{ \ell_{x_u},n_{x_u}\right\}_{x_u\in \mathbb{R}}$ is orthonormal, that for every fixed $x_u\in\mathbb{R}$, the scalar curvature of the Riemannian metric $q_{x_u}$ satisfies $\mathrm{Scal}^{q_{x_u}} = -2\lambda^2$ and hence $(X,q_{x_u})$ is a hyperbolic Riemann surface. Furthermore, by the previous equation we have:
\begin{equation*}
\dd_X n_{x_u} = \lambda n_{x_u}\wedge \ell_{x_u}\, .
\end{equation*}

\noindent
Since $2\lambda\ell_{x_u} = - \dd_X\cF_{x_u}$, the previous equation is equivalent to:
\begin{equation*}
\dd_X (e^{-\frac{1}{2} \cF_{x_{u}}}n_{x_u}) = 0\, ,
\end{equation*}

\noindent
whence there exists a family $\left\{ \omega_{x_u}\right\}_{x_u\in \mathbb{R}}$ of closed one-forms such that:
\begin{equation*}
\omega_{x_u} = e^{-\frac{1}{2} \cF_{x_{u}}}n_{x_u}\, , \qquad \forall\,\, x_u\in \mathbb{R}\, .
\end{equation*}

\noindent
Therefore:
\begin{equation*}
q_{x_u} = \ell_{x_u}\otimes \ell_{x_u} + n_{x_u}\otimes n_{x_u} = q_{x_u} = \frac{1}{4\lambda^2} \dd_X \cF_{x_{u}}\otimes \dd_X\cF_{x_{u}} + e^{\cF_{x_u}} \omega_{x_u}\otimes \omega_{x_u}\, ,
\end{equation*}

\noindent
and hence we conclude.
\end{proof}

\begin{remark}
If $X = \mathbb{R}^2$ then there exists a family of functions $\left\{ \cG_{x_u}\right\}_{x_u\in \mathbb{R}}$ such that:
\begin{equation*}
	e^{-\frac{1}{2} \cF_{x_{u}}}n_{x_u} = \dd_X \cG_{x_{u}}\, ,
\end{equation*}

\noindent
which immediately implies that:
\begin{equation*}
q_{x_u} = \frac{1}{4\lambda^2} \dd_X \cF_{x_{u}}\otimes \dd_X\cF_{x_{u}} + e^{\cF_{x_u}} \dd_X\cG_{x_u}\otimes \dd_X\cG_{x_u}\, .
\end{equation*}

\noindent
On the other hand, if $X=\mathbb{R}\times S^1$ then $H^1(X,\mathbb{Z}) = \mathbb{Z}$ and there exists a family of constants $\left\{ c_{x_u}\right\}_{x_u\in \mathbb{R}}$ and a family of functions $\left\{ \cG_{x_u}\right\}_{x_u\in \mathbb{R}}$ such that:
\begin{equation*}
e^{-\frac{1}{2} \cF_{x_{u}}}n_{x_u} = c_{x_u}\,\omega + \dd_X \cG_{x_{u}}\, , 
\end{equation*}

\noindent
where $\omega$ is a volume form on $S^1$. This parametrizes $\left\{q_{x_u}\right\}$ in terms of a family of pairs of functions $\left\{ \cF_{x_{u}} , \cG_{x_{u}}\right\}_{x_u\in \mathbb{R}}$ and a family of constants $\left\{ c_{x_{u}}\right\}_{x_u\in \mathbb{R}}$.
\end{remark}

\noindent
Since $X$ is either diffeomorphic to $\mathbb{R}^2$ or $\mathbb{R}\times S^1$, the uniformization theorem for Riemann surfaces yields the following result.

\begin{corollary}
Let $(M,g,\varepsilon)$ be a standard supersymmetric Kundt configuration. For every fixed $x_u\in \mathbb{R}$ the pair $(X,q_{x_u})$ is isometric to an elementary hyperbolic surface, namely it is isometric to either the Poincar\'e upper space, to a hyperbolic cylinder or to a parabolic cylinder, in all cases of scalar curvature $-2\lambda^2$.
\end{corollary}

\noindent
More explicitly, for every $x_u\in \mathbb{R}$ the Riemann surface $(X,q_{x_u})$ is locally isometric to the model:
\begin{equation*}
(\mathbb{R}^2,q = \frac{\dd\rho\otimes \dd\rho}{4\lambda^2} + e^{\rho} \dd w\otimes \dd w)\, ,
\end{equation*}

\noindent
which in turn is isometric to the Poincar\'e upper space of scalar curvature $-2\lambda^2$. The hyperbolic cylinder is obtained from $(\mathbb{R}^2,q)$ via quotient by the cyclic group generated by $(\rho,w) \mapsto (\rho - 2 l , e^l w)$, $l>0$. The parabolic cylinder is on the other hand obtained instead via taking the quotient by the cyclic group generated by $(\rho,w) \mapsto (\rho,w+1)$. In both cases, the globally defined one-form $\dd\rho \in \Omega^1(\mathbb{R}^2)$ is invariant and therefore descends to the quotient and defines the constant-norm gradient that Lemma \ref{lemma:iffRKS} requires for an adapted Killing spinor to exist on $\mathbb{R}^2\times X$.

\begin{lemma}
\label{lemma:alphabeta}
Let $(M,g,\varepsilon)$ be a standard supersymmetric Kundt configuration. Then the family of one-forms $\left\{ \alpha_{x_{u}}\right\}_{x_u\in\mathbb{R}}$ can be expressed as:
\begin{eqnarray}
\alpha_{x_u} = e^{\cF_{x_{u}}} \beta_{x_{u}}\, , 
\end{eqnarray}
	
\noindent
where $\left\{ \beta_{x_{u}}\right\}_{x_u\in\mathbb{R}}$ is a family of one-forms satisfying the following differential equation:
\begin{equation}
\label{eq:betaequation}
\dd_X \beta_{x_{u}} = \frac{e^{-\cF_{x_{u}}}}{4\lambda^2} \langle \partial_{x_u} \ast_{q_{x_u}}\dd_X \cF_{x_{u}} + \ast_{q_{x_u}}\partial_{x_u}\dd_X\cF_{x_{u}} , \dd_X \cF_{x_{u}} \rangle_{q_{x_u}} \nu_{q_{x_u}}\, .
\end{equation}

\noindent
This equation always has solutions, and the solution space is an affine space modeled on the infinite-dimensional vector space of closed 1-forms on $X$.
\end{lemma}

\begin{proof}
Let $(M,g,\varepsilon)$ be a standard supersymmetric Kundt configuration. Then, the family of one-forms $\left\{\alpha_{x_{u}}\right\}_{x_u\in\mathbb{R}}$ satisfies equation \eqref{eq:evolution2dX}. Plugging $\alpha_{x_{u}} = e^{\cF_{x_{u}}} \beta_{x_{u}}$ into equation \eqref{eq:evolution2dX} and isolating $\beta_{x_u}$ we obtain equation \eqref{eq:betaequation}. The last statement of the lemma follows from Proposition \ref{prop:standardqxu}, which proves that either $X=\mathbb{R}^2$ or $X=\mathbb{R}\times S^1$ whence $H^2(X,\mathbb{R}) = 0$.  

\end{proof}

\noindent
We can explicitly solve equation \eqref{eq:betaequation} in terms of local coordinates $(y,z)$, which can be taken to be global if $X=\mathbb{R}^2$. Using these coordinates, the most general solution can be found to be:
\begin{eqnarray}
\label{eq:standardalphaxu}
\beta_{x_u} =  \left (\gamma_{x_u}(y)+\int_0^z \left[ \partial_y  \kappa_{x_u}(y,z')-   \Upsilon_{x_u}(y,z') \right ]\, \dd z'  \right )\,\dd y+ \kappa_{x_u}(y,z) \, \dd z   \nonumber
\end{eqnarray}

\noindent
for families of functions $\left\{\kappa_{x_{u}}(y,z),\gamma_{x_{u}}(y)\right\}_{x_u\in \mathbb{R}}$ depending on the indicated variables, and where we have defined:
\begin{equation*}
\Upsilon_{x_u} := \frac{e^{-\cF_{x_{u}}}}{4\lambda^2 } \langle \partial_{x_u} \ast_{q_{x_u}} \dd_X \cF_{x_u} + \ast_{q_{x_u}} \partial_{x_u} \dd_X \cF_{x_u}, \dd_X \cF_{x_u} \rangle_{q_{x_u}} q_{x_u}^{\frac{1}{2}}\, ,
\end{equation*}

\noindent
with the understanding that the square root $q^{\frac{1}{2}}_{x_u}$ of the determinant of $q_{x_u}$ is to be taken in the coordinates $(y,z)$. This provides an explicit parametrization of the infinite-dimensional space of local solutions.

\begin{theorem}
\label{thm:conformallyBrinkmann}
A triple $(\mathbb{R}^2\times X,g, \varepsilon)$ is a standard supersymmetric Kundt configuration if and only if there exist a family $\left\{\omega_{x_u}\right\}_{x_u\in \mathbb{R}}$ of closed one-forms on $X$ in terms of which the metric $g$ reads:    
\begin{eqnarray}
\label{eq:thmcondition1}
g = \cH_{x_u} \dd x_u\otimes \dd x_u +   e^{\cF_{x_u}} \dd x_u\odot ( \dd x_v + \beta_{x_u}) + \frac{1}{4\lambda^2} \dd_X \cF_{x_{u}}\otimes \dd_X\cF_{x_{u}} + e^{\cF_{x_u}} \omega_{x_u}\otimes \omega_{x_u}
\end{eqnarray}

\noindent
where $\left\{ \beta_{x_u}\right\}_{x_u\in \mathbb{R}}$ is a family of one-forms on $X$ satisfying the following equation:
\begin{equation}
\label{eq:thmcondition2}
\dd_X \beta_{x_{u}} = \frac{e^{-\cF_{x_{u}}}}{4\lambda^2} \langle \partial_{x_u} \ast_{q_{x_u}}\dd_X \cF_{x_{u}} + \ast_{q_{x_u}}\partial_{x_u}\dd_X\cF_{x_{u}} , \dd_X \cF_{x_{u}} \rangle_{q_{x_u}} \nu_{q_{x_u}}\, .
\end{equation}
 
\noindent
In particular, for every $x_u\in\mathbb{R}$ the pair:
\begin{equation*}
(X,q_{x_u} = \frac{1}{4\lambda^2} \dd_X \cF_{x_{u}}\otimes \dd_X\cF_{x_{u}} + e^{\cF_{x_u}} \omega_{x_u}\otimes \omega_{x_u})
\end{equation*}

\noindent
is an elementary hyperbolic surface of scalar curvature $-2\lambda^2$, and therefore diffeomorphic to either $\mathbb{R}^2$ or $\mathbb{R}\times S^1$.
\end{theorem}

\begin{remark}
As explained below Lemma \ref{lemma:alphabeta}, equation \eqref{eq:thmcondition2} admits a plethora of solutions.
\end{remark}

\begin{proof}
The \emph{only if} direction follows from Proposition \ref{prop:standardqxu} and Lemma \ref{lemma:alphabeta}. For the \emph{if} direction, consider a standard conformally Brinkmann space-time $(M=\mathbb{R}^2\times X,g)$ whose metric $g$ is as prescribed in the statement. To prove that such space-time admits an adapted Killing spinor it is enough to prove that the necessary and sufficient conditions of Lemma \ref{lemma:iffRKS} are satisfied. For each fixed $x_u\in \mathbb{R}$, the functions $\cF_{x_{u}}$ and $\cG_{x_{u}}$ have linearly independent differentials at every point in $X$ (otherwise $g$ would be degenerate) and therefore provide local coordinates $(\rho,w)$ on $X$ in terms of which the metric $q_{x_u}$ reads:
\begin{equation*}
q_{x_u} = \frac{1}{4\lambda^2} \dd\rho\otimes \dd\rho + e^{\rho} \dd w\otimes \dd w\, ,
\end{equation*} 

\noindent
which is isometric to the Poincar\'e metric of Ricci curvature $-\lambda^2$. In these coordinates, we have $\dd_X \cF_{x_{u}} = \dd \rho$ and a quick computation shows that:
\begin{equation*}
\vert\dd\cF_{x_{u}}\vert^2_{q_{x_u}} = 4\lambda^2 \quad \forall\,\, x_u\in\mathbb{R}
\end{equation*}

\noindent
whence the first equation in \eqref{eq:constraint2dX} is satisfied. Furthermore, we have:
\begin{equation*}
\nabla^{q_{x_u}} \dd_X\cF_{x_{u}} = \nabla^{q_{x_u}}\dd\rho = 2\lambda^2 e^\rho \dd w \otimes \dd w\, ,
\end{equation*}

\noindent
whence the second equation in \eqref{eq:constraint2dX} also follows. Finally, equation \eqref{eq:thmcondition2} is equivalent to equation \eqref{eq:evolution2dX} by Lemma \ref{lemma:alphabeta} and hence we conclude. 
\end{proof}

\noindent
Using the fact that if $M$ is simply connected then $X=\mathbb{R}^2$ and hence contractible, we obtain the following refinement of equation \ref{eq:thmcondition1} in Theorem \ref{thm:conformallyBrinkmann}.

\begin{corollary}
Every choice of families of functions $\left\{\cF_{x_u},\cG_{x_u}\right\}_{x_u\in \mathbb{R}}$ on $\mathbb{R}^2$ with everywhere linearly independent differentials determines a standard supersymmetric Kundt configuration on $\mathbb{R}^4$ with metric:
\begin{equation}
\label{eq:metricBspinorgeneral}
g = \cH_{x_u} \dd x_u\otimes \dd x_u +   e^{\cF_{x_u}} \dd x_u\odot ( \dd x_v + \beta_{x_u}) + \frac{1}{4\lambda^2} \dd_X \cF_{x_{u}}\otimes \dd_X\cF_{x_{u}} + e^{\cF_{x_u}} \dd_X\cG_{x_u}\otimes \dd_X\cG_{x_u} 
\end{equation} 
	
\noindent
for a choice of $\left\{\beta_{x_{u}}\right\}_{x_u\in \mathbb{R}}$ as prescribed in equation \eqref{eq:thmcondition2}. Conversely, every simply connected  standard supersymmetric Kundt configuration can be constructed in this way.
\end{corollary}

\noindent
Theorem \ref{thm:conformallyBrinkmann} characterizes as well the local isometry type of every supersymmetric Kundt configuration.

\begin{corollary}
Every four-dimensional space-time $(M,g)$ admitting a real Killing spinor is locally isometric to an open set of $\mathbb{R}^4$ equipped with the metric:
\begin{equation*}
g = \cH_{x_u} \dd x_u\otimes \dd x_u +   e^{\cF_{x_u}} \dd x_u\odot ( \dd x_v + \beta_{x_u}) + \frac{1}{4\lambda^2} \dd_X \cF_{x_{u}}\otimes \dd_X\cF_{x_{u}} + e^{\cF_{x_u}} \dd_X\cG_{x_u}\otimes \dd_X\cG_{x_u} \, ,
\end{equation*} 
	
\noindent
for a choice of families of functions $\left\{\cF_{x_u},\cG_{x_u}\right\}_{x_u\in \mathbb{R}}$ and of one-forms $\left\{\beta_{x_{u}}\right\}_{x_u}$ as prescribed in equation \eqref{eq:thmcondition2}. 
\end{corollary}

\begin{example}
Assume that $X = \mathbb{R}^2$ with Cartesian coordinates $(y_1,y_2)$, take $\lambda = \frac{1}{2}$ and write $\left\{q_{x_u}\right\}_{x_u\in \mathbb{R}}$ as follows:
\begin{equation*}
q_{x_u} =  \dd_X \cF_{x_{u}}\otimes \dd_X\cF_{x_{u}} + e^{\cF_{x_u}} \dd_X\cG_{x_u}\otimes \dd_X\cG_{x_u}\, ,
\end{equation*}

\noindent
in terms of families of functions $\left\{\cF_{x_{u}} ,\cG_{x_{u}}\right\}_{x_u\in \mathbb{R}}$ on $X$. Equation \eqref{eq:thmcondition2} can be equivalently written as follows:
\begin{equation}
\label{eq:thmcondition2alt}
\dd_X \beta_{x_{u}} =   (   \partial_{x_u}  \dd_X \cG_{x_{u}} + e^{-\cF_{x_{u}}/2}   \ast_{q_{x_u}}\partial_{x_u}\dd_X\cF_{x_{u}} ) \wedge \dd_X \cG_{x_{u}} 
\end{equation}

\noindent
where we have used that $\ast_{q_{x_u}}\dd_X\cF_{x_{u}} =  e^{\cF_{x_{u}}/2} \dd_X\cG_{x_{u}}$. Assume that $\{\cF_{x_{u}}, \cG_{x_{u}}\}_{x_u \in \mathbb{R}}$ are such that:
\begin{equation*}
\dd_X \cF_{x_{u}} = a_{x_u} \dd y_1 \, , \qquad \dd_X \cG_{x_{u}} = f_{x_u} \dd y_1  + b_{x_u} \dd y_2\, , \qquad x_u\in \mathbb{R}\, ,
\end{equation*}

\noindent
$\left\{ a_{x_u} , b_{x_u}, f_{x_u} \right\}_{x_u\in\mathbb{R}}$ being families of constant functions on $X = \mathbb{R}^2$. In particular:
\begin{equation*}
q_{x_u} =   (a_{x_u}^2 + f_{x_u}^2 e^{a_{x_u} y_1 + k_{x_u}}) \dd y_1 \otimes \dd y_1 +  e^{a_{x_u} y_1 + k_{x_u}} (b_{x_u} f_{x_u} \dd y_1\odot \dd y_2 + b_{x_u}^2 \dd y_2 \otimes \dd y_2)\, ,
\end{equation*} 

\noindent
where $\left\{ k_{x_u} \right\}_{x_u\in\mathbb{R}}$ is a family of constants. A quick computation shows that:
\begin{equation*}
 e^{-\cF_{x_{u}}/2}   \ast_{q_{x_u}}\partial_{x_u}\dd_X\cF_{x_{u}} = \partial_{x_u}\log(a_{x_u}) \dd_X\cG_{x_{u}}\, ,
\end{equation*}

\noindent
where we assume that the family of constants $\left\{ a_{x_u}\right\}_{x_u\in\mathbb{R}}$ are strictly positive. Hence:

\begin{equation*}
\dd_X \beta_{x_{u}} =  \partial_{x_u}  \dd_X \cG_{x_{u}}  \wedge \dd_X \cG_{x_{u}} = (b_{x_u}\partial_{x_u}f_{x_u} - f_{x_u}\partial_{x_u}b_{x_u}) \dd y_1 \wedge \dd y_2\, .
\end{equation*}

\noindent
Solutions to the previous equation can be easily found by direct inspection. For instance:
\begin{equation*}
\beta_{x_{u}} = \frac{1}{2}(b_{x_u}\partial_{x_u}f_{x_u} - f_{x_u}\partial_{x_u}b_{x_u}) (y_1   \dd y_2	- y_2   \dd y_1)\, ,
\end{equation*}

\noindent
which yields the following four-dimensional metric $g$ on $\mathbb{R}^4$:
\begin{eqnarray*}
& g = \cH_{x_u} \dd x_u\otimes \dd x_u +   \frac{1}{2} e^{a_{x_u} y_1 + k_{x_u}} \dd x_u\odot ( \dd x_v + (b_{x_u}\partial_{x_u}f_{x_u} - f_{x_u}\partial_{x_u}b_{x_u}) (y_1   \dd y_2	- y_2   \dd y_1))\\
& + (a_{x_u}^2 + f_{x_u}^2 e^{a_{x_u} y_1 + k_{x_u}}) \dd y_1 \otimes \dd y_1 +  e^{a_{x_u} y_1 + k_{x_u}} (b_{x_u} f_{x_u} \dd y_1\odot \dd y_2 + b_{x_u}^2 \dd y_2 \otimes \dd y_2)\, ,
\end{eqnarray*} 

\noindent
This provides an example of supersymmetric Kundt configuration for which the crossed term $\beta_{x_{u}}$ is not trivial and, in particular, not closed. 
\end{example}

For future applications it is convenient to present the metric $g$ occurring in equation \eqref{eq:thmcondition1} of Theorem \ref{thm:conformallyBrinkmann} in an alternative equivalent form. Using the notation of Theorem \ref{thm:conformallyBrinkmann}, define:
\begin{equation*}
\cY_{x_u}  = e^{-\cF_{x_{u}}/2} \, , \quad x_u \in \mathbb{R}\, .	
\end{equation*}

\noindent
This defines a family $\left\{\cY_{x_u} \right\}_{x_u\in\mathbb{R}}$ of strictly positive functions on $X$. Substituting this expression into equation \eqref{eq:thmcondition1} and relabeling some of the symbols adequately we obtain the following equivalent expression for $g$:
\begin{equation}
\label{eq:gmetricalternative}
g = \frac{1}{\lambda^2 \cY_{x_u}^2} (\cK_{x_u}  \dd x_u\otimes \dd x_u + \dd x_u\odot (\dd x_v +   \lambda^2 \beta_{x_u}) + \dd_X \cY_{x_{u}}\otimes \dd_X\cY_{x_{u}} +  \omega_{x_u} \otimes \omega_{x_u})
\end{equation}

\noindent
In this form it becomes apparent that supersymmetric Kundt configurations provide a vast generalization of Siklos space-times, which in turn can be interpreted as a deformation of the AdS$_4$ space-time. Indeed, assume that $(M,g,\varepsilon)$ is a supersymmetric Kundt configuration for which $\left\{ \cF_{x_{u}} \right\}_{x_u\in \mathbb{R}}$ and  $\left\{ \cG_{x_{u}} \right\}_{x_u\in \mathbb{R}}$ are both independent of the coordinate $x_u$. Then, by equation \eqref{eq:gmetricalternative} it is clear that there exists local coordinates in which the metric $g$ reads:
\begin{equation*}
g = \frac{1}{\lambda^2 y^2} (\cK_{x_u}  \dd x_u\otimes \dd x_u + \dd x_u\odot (\dd x_v +  \lambda^2 \beta_{x_u}) + \dd y \otimes \dd y + \dd w\otimes \dd w)\, ,
\end{equation*}

\noindent
In addition equation \eqref{eq:thmcondition2} implies in this case that $\beta_{x_{u}}$ is closed on $X$, whence locally exact, a fact that can be used to redefine $\cK_{x_{u}}$ as well as the coordinate $x_v$ in order to absorb the one-form $\beta_{x_{u}}$ in such a way that $g$ locally reads:
\begin{equation*}
g = \frac{1}{\lambda^2 y^2} (\cK_{x_u}  \dd x_u\otimes \dd x_u + \dd x_u\odot \dd x_v + \dd y \otimes \dd y + \dd w\otimes \dd w)\, .
\end{equation*}

\noindent
This is precisely the local four-dimensional metric constructed by Siklos in \cite{Siklos}, which defines what are nowadays called \emph{Siklos space-times} or \emph{Siklos gravitational waves}. The latter describe exact idealized gravitational waves moving through anti-de Sitter space-time \cite{GibbonsRuback,Podolsky:1997ik}. The fact that every Siklos space-time admits real Killing spinors was explicitly noticed in \cite{GibbonsRuback}. Therefore, supersymmetric Kundt configurations  provide a broad generalization of the Siklos class of space-times that also admits Killing spinors and reduces to the latter in certain special cases.


\section{Real Killing spinorial flows on three-dimensional Cauchy hypersurfaces}
\label{sec:GloballyHyperbolic}


By the results of Bernal and S\'anchez \cite{Bernal:2003jb}, if $(M,g)$ is globally hyperbolic then there exists an oriented three-manifold $\Sigma$ such that $(M,g)$ has the following isometry type:
\begin{equation}
\label{eq:globahyp}
(M,g) = (\cI\times \Sigma, -\beta^2_t \dd t\otimes \dd t + h_t)\, ,
\end{equation}

\noindent
where $t$ is the Cartesian coordinate on the connected interval $\cI\subset \mathbb{R}$, $\left\{\beta_t\right\}_{t\in\cI}$ is a family of positive functions and $\left\{ h_t\right\}_{t\in\cI}$ is a family of complete Riemannian metrics on $\Sigma$. We will denote by $\mathrm{F}(\Sigma)$ the frame bundle of $\Sigma$. Submanifolds of the form $\left\{ t\right\}\times \Sigma \hookrightarrow M$ are Cauchy surfaces of $(M,g)$. We will consider hereafter that the identification \eqref{eq:globahyp} has been fixed. Set:
\begin{equation*}
\Sigma_t := \left\{ t\right\}\times \Sigma \hookrightarrow M\, , \qquad \Sigma := \left\{ 0\right\}\times \Sigma \hookrightarrow M\, ,
\end{equation*}

\noindent
and let $\mathfrak{t}_t = \beta_t\, \dd t$ be outward-pointing unit time-like one-form orthogonal to $T^{\ast}\Sigma_t$ for every $t\in \cI$, which trivializes the normal bundle of $T^{\ast}\Sigma_t$. Take $\Sigma\hookrightarrow M$ together with the Riemannian metric:
\begin{equation*}
h := h_0\vert_{T\Sigma\times T\Sigma}\, ,
\end{equation*}

\noindent
to be the Cauchy hypersurface of $(M,g)$. The associated \emph{shape operator} or scalar second fundamental form $\Theta_t$ is defined by:
\begin{equation*}
\Theta_t  := \nabla^g \mathfrak{t}_t\vert_{T\Sigma_t\times T\Sigma_t}=- \frac{1}{2\beta_t} \partial_t h_t \in \Gamma(T^{\ast}\Sigma_t\odot T^{\ast}\Sigma_t)\, .
\end{equation*}

\noindent
Given a parabolic pair $(u,[l])$, we perform the following decomposition:
\begin{equation}
\label{eq:orthosplit}
u = u^0_t\, \mathfrak{t}_t + u^{\perp}_t\, , \qquad l = l^0_t\, \mathfrak{t}_t + l^{\perp}_t\in [l]\, ,
\end{equation}

\noindent
where $l\in[l]$ is a representative, the superscript $\perp$ denotes orthogonal projection to $T^{\ast}\Sigma_t$ and:
\begin{equation*}
u^0_t = - g(u,\mathfrak{t}_t)\, , \qquad l^0_t = - g(l,\mathfrak{t}_t)\, .
\end{equation*}
 
\begin{lemma} 
\label{lemma:globhyperspinor}
A globally hyperbolic four-manifold $(M,g)  = (\cI\times \Sigma, -\beta^2_t \dd t\otimes \dd t + h_t)$ admits a Killing spinor if and only if there exists a family of orthogonal one-forms $\left\{ u^{\perp}_t, l^{\perp}_t \right\}_{t\in\cI}$ on $\Sigma$ satisfying:
\begin{eqnarray}
\label{eq:globhyperspinorI}
& \partial_t u^{\perp}_t + \beta_t \Theta_t(u^{\perp}_t) = u^0_t\dd \beta_t + \lambda u^0_t \beta_t l^{\perp}_t\, , \quad  u^0_t \partial_t l^{\perp}_t +\beta_t  u^0_t \Theta_t(l^{\perp}_t)  = -(\dd \beta_t (l^{\perp}_t) + \lambda \beta_t) u^{\perp}_t\,  \\
\label{eq:globhyperspinorII}
& \nabla^{h_t} u^{\perp}_t + u^0_t \Theta_t = \lambda u^{\perp}_t \wedge l^{\perp}_t \, , \quad u^0_t \nabla^{h_t} l^{\perp}_t = \Theta_t(l^{\perp}_t)\otimes u^{\perp}_t + \lambda u^0_t (l^{\perp}_t \otimes l^{\perp}_t - h_t)\, 
\end{eqnarray}

\noindent
as well as:
\begin{equation}
\label{eq:restrictionul}
(u^0_t)^2 = \vert u^{\perp}_t \vert^2_{h_t}\, , \qquad \vert l^{\perp}_t\vert^2_{h_t} = 1\, ,
\end{equation}

\noindent
where $\vert u_t^{\perp}\vert^2_{h_t} = h_t(u^{\perp}_t , u^{\perp}_t)$ and $\vert l_t^{\perp}\vert^2_{h_t} = h_t(l^{\perp}_t , l^{\perp}_t)$. In particular, $\partial_t u^0_t = \dd \beta_t(u^{\perp}_t)$ and $\dd u^0_t + \Theta(u^{\perp}_t) + \lambda u^0_t l^{\perp}_t = 0$. If all conditions in \eqref{eq:globhyperspinorI}, \eqref{eq:globhyperspinorII} and \eqref{eq:restrictionul} are met, the corresponding Killing parabolic pair $(u,[l])$ is given by:
\begin{equation*}
u = u^0_t \mathfrak{t}_t +  u_t^{\perp}\, ,  \qquad  [l] = [ l_t^{\perp}]\, ,
\end{equation*}

\noindent
and defines a real Killing spinor, unique modulo a sign.
\end{lemma}

\begin{proof}
By Proposition \ref{prop:Killingspinorgeneral}, $(M,g)$ admits a real Killing spinor if and only if it admits a parabolic pair satisfying equations \eqref{eq:RKSspinorgeneral}. Assume $(u,[l])$ is such a parabolic pair and write $u = u^0_t\, \mathfrak{t}_t + u^{\perp}_t$. By Lemma \ref{lemma:dependencerep} it is possible to find a representative $l\in [l]$ such that:
\begin{equation*}
l =  l^{\perp}_t\in \Omega^1(\Sigma_t)\, , \qquad t\in \cI\, ,
\end{equation*}

\noindent
that is, such that $l$ is purely spatial. Using this representative, equations \eqref{eq:RKSspinorgeneral} turn out to be equivalent to:
\begin{eqnarray*}
&\nabla^g_{\partial_t} u = \lambda u^0_t \beta_t l^{\perp}_t\, , \quad \nabla^g u \vert_{T\Sigma_t\times TM}  = \lambda (u^{\perp}_t \wedge l^{\perp}_t -   u^0_t\, l^{\perp}_t \otimes \mathfrak{t}_t)\, , \\ 
&\nabla^g_{\partial_t} l^{\perp}_t = \kappa(\partial_t)\, u +\lambda \beta_t \mathfrak{t}_t\, , \quad \nabla^g l^{\perp}_t \vert_{T\Sigma_t\times TM}  = \kappa^{\perp}_t \otimes u + \lambda (l^{\perp}_t \otimes l^{\perp}_t - h_t)\, ,  
\end{eqnarray*}

\noindent
where we have decomposed $\kappa = \kappa^0_t\, \mathfrak{t}_t + \kappa^{\perp}_t$ in terms of a family of functions $\left\{ \kappa^0_t\right\}_{t\in \cI}$ and one-forms $\left\{ \kappa^{\perp}_t\right\}_{t\in \cI}$ on $\Sigma$. We compute:
\begin{eqnarray*}
&\nabla^g_{\partial_t} u   =\partial_t u^{\perp}_t + \beta_t \Theta_t(u^{\perp}_t) - u^0_t\dd \beta_t +   \partial_t u^0_t  \mathfrak{t}_t -  \dd\beta_t(u^{\perp}_t)  \mathfrak{t}_t = \lambda u^0_t \beta_t l^{\perp}_t\\
&\nabla^g u\vert_{T\Sigma_t\times TM} = \nabla^{h_t} u^{\perp}_t + u^0_t \Theta_t + \dd u^0_t \otimes \mathfrak{t}_t + \Theta_t(u^{\perp}_t)\otimes \mathfrak{t}_t = \lambda (u^{\perp}_t \wedge l^{\perp}_t -   u^0_t\, l^{\perp}_t \otimes \mathfrak{t}_t)\, , \\
&\nabla^g_{\partial_t} l^{\perp}_t =   \partial_t l^{\perp}_t + \beta_t \Theta_t(l^{\perp}_t) - \dd \beta_t (l^{\perp}_t) \mathfrak{t}_t  = \kappa(\partial_t)\, u + \lambda \beta_t \mathfrak{t}_t\, ,\\
&\nabla^g l^{\perp}_t \vert_{T\Sigma_t\times TM} = \nabla^{h_t} l^{\perp}_t +\Theta_t(l^{\perp}_t)\otimes \mathfrak{t}_t  =\kappa^{\perp}_t \otimes u + \lambda (l^{\perp}_t \otimes l^{\perp}_t - h_t)\, .
\end{eqnarray*}

\noindent
Solving for $\kappa$ in the last two previous equations:
\begin{equation*}
\kappa(\partial_t) = -\frac{1}{u^0_t}( \lambda \beta_t + \dd \beta_t(l^{\perp}_t))\, , \qquad \kappa^{\perp}_t = \frac{1}{u^0_t} \Theta_t(l^{\perp}_t)\, ,
\end{equation*}

\noindent
Substituting into the equations for the covariant derivative of $l^{\perp}_t$, we obtain all equations \eqref{eq:globhyperspinorI} and \eqref{eq:globhyperspinorII} as well as the following two equations:
\begin{equation*}
\partial_t u^0_t = \dd \beta_t(u^{\perp}_t)\, , \qquad \dd u^0_t + \Theta(u^{\perp}_t) + \lambda u^0_t l^{\perp}_t = 0\, .
\end{equation*}

\noindent 
These equations are not independent and can be obtained by manipulating the time and exterior derivatives of $(u^0_t)^2 = \vert u^{\perp}_t \vert^2_{h_t}$ upon the use of equations \eqref{eq:globhyperspinorI} and \eqref{eq:globhyperspinorII}. 

\noindent
The converse may be seen to follow by construction.
\end{proof}

\noindent
Equations \eqref{eq:globhyperspinorI}, \eqref{eq:globhyperspinorII} and \eqref{eq:restrictionul} represent the necessary and sufficient conditions for a globally hyperbolic Lorentzian four manifold $M$ to admit a real Killing spinor. The variables of these equations are tuples of the form:
\begin{equation}
\label{eq:originaltuple}
\left\{ \beta_t ,  h_t , u^0_t , u^{\perp}_t , l^{\perp}_t\right\}_{t\in \cI}\, .
\end{equation}

\begin{theorem}
\label{thm:Killingspinorflow}
An oriented globally hyperbolic Lorentzian four-manifold $(M,g)$ admits a real Killing spinor with Killing constant $\lambda \in \mathbb{R}$ if and only if $M = \mathbb{R}\times \Sigma$, where $\Sigma$ is an oriented three-manifold equipped with a family of strictly positive functions $\left\{ \beta_t\right\}_{t\in \cI}$ on $\Sigma$ and a family $\left\{ \fre^t \right\}_{t\in \cI}$ of sections of $\mathrm{F}(\Sigma)$ satisfying:
\begin{eqnarray}
\label{eq:globhyperspinorflowI}
& \partial_t e^t_u + \dd\beta_t(e^t_u) e^t_u  + \beta_t \Theta_t(e^t_u) =   \dd\beta_t + \lambda \beta_t e_l^t\, , \quad \partial_t e^t_l +  \beta_t \Theta_t(e^t_l) =  -(\dd\beta_t(e^t_l) + \lambda \beta_t) e_u^t\\
& \partial_t e^t_n +  \beta_t \Theta_t(e^t_n) =  - \dd\beta_t(e^t_n) e_u^t\, , \quad \partial_t(\Theta_t(e^t_u ))+ \lambda \, \partial_t e^t_l = - \dd (\dd \beta_t(e_u^t)) \label{eq:globhyperspinorflowII} \\
& \dd e^t_u=(\Theta_t(e^t_u) - \lambda e^t_l) \wedge e^t_u\, , \quad \dd e^t_l=\Theta_t(e^t_l) \wedge e^t_u  \, , \quad \dd e^t_n + \lambda e^t_l\wedge e^t_n = \Theta_t(e^t_n) \wedge e^t_u\label{eq:globhyperspinorflowIII}\\
& [\Theta_t(e^t_u )+ \lambda  e^t_l]=0\in H^1(\Sigma,\mathbb{R})
\label{eq:globhyperspinorflowIV}
\end{eqnarray}

\noindent
where $\fre^t = (e_u^t , e_l^t, e_n^t) \colon \Sigma \to \mathrm{F}(\Sigma)$ and:
\begin{equation*}
h_{\fre^t} = e_u^t\otimes e_u^t + e_l^t\otimes e_l^t + e_n^t\otimes e_n^t\, , \qquad \Theta_t = - \frac{1}{2\beta_t}\partial_t h_{\fre^t}\, .
\end{equation*}

\noindent
In such a case, the globally hyperbolic metric $g$ is given by:
\begin{equation*}
	g = - \beta_t^2 \dd t\otimes \dd t + h_{\fre^t} \, ,
\end{equation*} 

\noindent
where $t$ is the Cartesian coordinate in the splitting $M = \mathbb{R}\times \Sigma$.
\end{theorem}
 
\begin{proof}
Lemma \ref{lemma:globhyperspinor} asserts that a globally hyperbolic Lorentzian four-manifold $(M,g)$ admits a real Killing spinor if and only if there exists a Cauchy surface $\Sigma \hookrightarrow M$ endowed with a tuple \eqref{eq:originaltuple} fulfilling equations \eqref{eq:globhyperspinorI}, \eqref{eq:globhyperspinorII} and \eqref{eq:restrictionul}. Let:
\begin{equation}
\left\{ \beta_t  ,   h_t  ,  u^0_t , u^{\perp}_t , l^{\perp}_t\right\}_{t\in \cI}\, ,
\end{equation}

\noindent
be a solution of the aforementioned equations. Write:
\begin{equation}
\label{eq:redfeuel}
e_u^t = \frac{u^{\perp}_t}{u^0_t}\, , \qquad e^t_l = l^{\perp}_t\, .
\end{equation}

\noindent
The pair $(e^t_u , e^t_l)$ defines a family of nowhere-vanishing and orthonormal one-forms on $\Sigma$. They can be naturally completed to a family of orthonormal coframes $\left\{  \fre^t = (e^t_u , e^t_l , e^t_n) \right\}_{t\in\cI}$ by introducing the following family of one-forms $\left\{ e_n^t\right\}_{t\in\cI}$:
\begin{equation*}
e_n^t := \ast_{h_t}(e^t_u \wedge e^t_l)\, , \qquad t\in \cI\, ,
\end{equation*} 

\noindent
where $\ast_{h_t}$ stands for the Hodge dual associated to the family of metrics $\left\{ h_t\right\}_{t\in\cI}$. Substituting equations \eqref{eq:redfeuel} into the first and second equations in \eqref{eq:globhyperspinorI} and massaging the time derivative of $\left\{ u^{\perp}_t\right\}_{t\in \mathbb{R}}$, equations \eqref{eq:globhyperspinorflowI} are derived. For $\left\{ e_n^t\right\}_{t\in\cI}$ we note that:
\begin{eqnarray*}
& 0 = \partial_t (h^{-1}_t (e_n^t, e_u^t )) = (\partial_t  h^{-1}_t) (e_n^t, e_u^t ) + h^{-1}_t (\partial_t e_n^t, e_u^t ) + h^{-1}_t (e_n^t, \partial_t e_u^t ) = 2\beta_t \Theta_t(e_n^t, e_u^t )\\
& + h^{-1}_t (\partial_t e_n^t, e_u^t ) + h^{-1}_t (e_n^t,    \dd\beta_t - \beta_t \Theta_t(e^t_u)) =  \beta_t \Theta_t(e_n^t, e_u^t )+ h^{-1}_t (\partial_t e_n^t, e_u^t ) + \dd\beta_t(e_n^t)\, ,\\
\
& 0 = \partial_t (h^{-1}_t (e_n^t, e_l^t )) = (\partial_t  h^{-1}_t) (e_n^t, e_l^t ) + h^{-1}_t (\partial_t e_n^t, e_l^t ) + h^{-1}_t (e_n^t, \partial_t e_l^t ) = 2\beta_t \Theta_t(e_n^t, e_u^t )\\
& + h^{-1}_t (\partial_t e_n^t, e_l^t ) + h^{-1}_t (e_n^t,  \beta_t \Theta_t(e^t_l)) =  \beta_t \Theta_t(e_n^t, e_u^t )+ h^{-1}_t (\partial_t e_n^t, e_l^t )\, ,\\
\end{eqnarray*}

\noindent
This implies the first equation in \eqref{eq:globhyperspinorflowII}. By virtue of Lemma \ref{lemma:globhyperspinor}, we observe that:
\begin{equation*}
\frac{\dd u^0_t}{u^0_t} + \Theta(e^{t}_u) + \lambda e^{t}_l = 0\, ,
\end{equation*}

\noindent
so $[\Theta_t(e^t_u )+ \lambda  e^t_l]=0\in H^1(\Sigma,\mathbb{R})$, which produces equation \eqref{eq:globhyperspinorflowIV}. Performing the time derivative of the previous equations we get:
\begin{equation*}
\dd \partial_t \log\vert u^0_t\vert + \partial_t\Theta(e^{t}_u) + \lambda\, \partial_t e^{t}_l = 0\, .
\end{equation*}

\noindent
Now, by Lemma \ref{lemma:globhyperspinor}, $\partial_t u^0_t = \dd\beta_t(u^{\perp}_t)$, so the previous equation gives rise to the second equation in \eqref{eq:globhyperspinorflowII}. As for equations \eqref{eq:globhyperspinorflowIII}, we compute:
\begin{equation*}
\nabla^{h_t} e^t_n = \nabla^h  \ast_{h_t} (e^t_u\wedge e^t_l) =    \ast_{h_t} (\nabla^{h_t} e^t_u\wedge e^t_l) + \ast_{h_t} ( e^t_u\wedge \nabla^{h_t} e^t_n) = \Theta_t(e^t_n)\otimes e^t_u + \lambda e^t_n\otimes e^t_l\, .
\end{equation*}

\noindent
Taking the anti-symmetric part of the previous equation and of equations \eqref{eq:globhyperspinorII}, we get all equations \eqref{eq:globhyperspinorflowIII}. Conversely, assume $\left\{ \fre^t , \beta_t\right\}_{t\in \cI}$ is a solution of equations \eqref{eq:globhyperspinorflowI}, \eqref{eq:globhyperspinorflowII}, \eqref{eq:globhyperspinorflowIII}   and \eqref{eq:globhyperspinorflowIV}. Set:
\begin{equation*}
h_{\fre^t} = e_u^t\otimes e_u^t + e_l^t\otimes e_l^t + e_n^t\otimes e_n^t\, , \qquad \Theta_t = - \frac{1}{2\beta_t}\partial_t h_{\fre^t}\, .
\end{equation*}

\noindent
Taking into account that $[\Theta_t(e^t_u )+ \lambda  e^t_l]=0$ in $H^1(\Sigma,\mathbb{R})$, there exists a family of functions $\left\{ \bar{\frf}_t \right\}_{t\in \mathbb{R}}$ satisfying that:
\begin{equation*}
\dd\bar{\frf}_t = - \Theta_t(e^t_u ) - \lambda e^t_l\, . 
\end{equation*}

\noindent
Performing the time derivative of the previous expression:
\begin{equation*}
\dd\partial_t\bar{\frf}_t = - \partial_t\Theta_t(e^t_u ) - \lambda \, \partial_t e^t_l\, .
\end{equation*}

\noindent
Comparing with the second equation in \eqref{eq:globhyperspinorflowII} we infer that:
\begin{equation*}
\dd\partial_t\bar{\frf}_t = \dd (\dd\beta_t(e^t_u))\, .
\end{equation*}

\noindent
so $\partial_t \bar{\frf}_t = \dd\beta_t(e^t_u) + c(t)$ for a certain function $c(t)$ depending solely on $t$. Define $\frf_t := \bar{\frf}_t - \int c(t) \dd t$, which satisfies $\partial_t \frf_t = \dd\beta_t(e^t_u)$ by construction. Additionally:
\begin{equation*}
\dd\partial_t\frf_t = - \partial_t\Theta_t(e^t_u ) - \lambda\, \partial_t e^t_l\, .
\end{equation*}

\noindent
Let $u^{\perp}_t := e^{\frf_t}\, e_u^t$ and $l^{\perp}_t := e^t_l$. Since $e_u^t$ and $e_u^t$ satisfy equations \eqref{eq:globhyperspinorflowI}:
\begin{align*}
\partial_t u^{\perp}_t + \beta_t \Theta_t(u^{\perp}_t) + (\dd\beta_t(e^t_u) -\partial_t \frf_t) u^{\perp}_t &= e^{\frf_t} \dd \beta_t + \lambda\, e^{\frf_t} \beta_t l^{\perp}_t\, , \\
e^{\frf_t} \partial_t l^{\perp}_t +\beta_t  e^{\frf_t}\Theta_t(l^{\perp}_t)  &= -(\dd \beta_t (l^{\perp}_t) + \lambda\, \beta_t) u^{\perp}_t\, ,
\end{align*}

\noindent
Carrying out the identification $u^0_t := e^{\frf_t}$ and using that $\partial_t \frf_t = \dd\beta_t(e^t_u)$, we get equations \eqref{eq:globhyperspinorI}. Now, equations \eqref{eq:globhyperspinorII} are derived by interpreting equations \eqref{eq:globhyperspinorflowIII} as the Cartan structure equations of the coframe $\fre^t$, considered to be orthonormal with respect to the metric $h_t = e^t_u\otimes e^t_u + e^t_l\otimes e^t_l + e^t_n\otimes e^t_n$. Finally, equations \eqref{eq:restrictionul} are satisfied by construction.
\end{proof}

\begin{definition}
\label{def:rksf}
Equations \eqref{eq:globhyperspinorflowI}, \eqref{eq:globhyperspinorflowII}, \eqref{eq:globhyperspinorflowIII} and \eqref{eq:globhyperspinorflowIV} are the \emph{(real) Killing spinorial flow equations}. A \emph{real Killing spinorial flow} is a family $\left\{ \beta_t ,\fre^t\right\}_{t\in \cI}$ of functions and coframes on $\Sigma$ satisfying the real Killing spinorial flow equations. A \emph{parallel} spinor flow is a Killing spinorial flow with $\lambda = 0$.
\end{definition}

\noindent
A globally hyperbolic Lorentzian four-manifold admits a real Killing spinor if and only if it admits a Cauchy surface carrying a Killing spinorial flow\footnote{Observe that the associated real Killing spinor $\varepsilon$ can be completely reconstructed from $\left\{ \beta_t , \fre^t\right\}_{t\in\cI}$, although its explicit expression is not of relevance for us. } $\left\{\beta_t , \fre^t\right\}_{t\in\cI}$. 

\begin{example}
Suppose that $\left\{\beta , \fre \right\} := \left\{\beta_t , \fre^t \right\}_{t\in\cI}$ is a stationary Killing spinorial flow. Then, $\Theta_t = 0$ and $\left\{\beta , \fre \right\}$ satisfies the following equations:
\begin{eqnarray}
\label{eq:staticflowI}
& \dd\beta(e_u) e_u  =   \dd\beta + \lambda \beta e_l\, , \quad  \dd\beta(e_l) + \lambda \beta = 0 \, , \quad \dd\beta_t(e_n) = 0\, , \quad \dd ( \dd \beta (e_u)) = 0\, ,\\
& \dd e_u=  \lambda e_u \wedge e_l\, , \quad \dd e_l= 0 \, , \quad \dd e_n = \lambda e_n\wedge e_l\, , \quad [\lambda e_l]=0\in H^1(\Sigma,\mathbb{R})\, . \label{eq:staticflowII}
\end{eqnarray}

\noindent
Assume that $\Sigma$ is simply connected and $\lambda \neq 0$ (the case $\lambda = 0$ is considered in \cite{Murcia:2020zig}). Then, equations \eqref{eq:staticflowI} are solved by setting:
\begin{equation}
\label{eq:betaexample1}
\dd \beta = c_1\, e_u - \lambda \beta\, e_l\, ,
\end{equation} 

\noindent
where $c_1\in \mathbb{R}$. Note that this becomes a differential equation for $\beta$. On the other hand, since the metric defined by $\fre$ is by assumption complete, the frame dual to $\fre$ is also complete. Therefore, equations \eqref{eq:staticflowII} immediately imply that $\fre$ defines the structure of a Lie group  $\G$ on $\Sigma$ which is unique modulo isomorphism. Comparing with \cite{Freibert,Milnor}, we readily conclude that $\G$ is isomorphic to the non-unimodular Lie group $\tau_{3,1}$ (defined as the unique connected and simply connected Lie group whose Lie algebra is precisely $\tau_{3,1}$) which as a manifold is diffeomorphic to $\mathbb{R}^3$. In particular, there exist global coordinates $(x,y,z)$ on $\G$ in terms of which the Riemannian metric $h_{\fre}$ reads:
\begin{equation*}
h_{\fre} = \dd x \otimes \dd x + e^{-2 \lambda x} (\dd y\otimes \dd y + \dd z \otimes \dd z)\, ,
\end{equation*}

\noindent
and $e_u = e^{-\lambda x} \dd y$, $e_l = \dd x$ and $e_n = e^{-\lambda x} \dd z$. In particular, equation \eqref{eq:betaexample1} becomes:
\begin{equation*}
\partial_x \beta = - \lambda \beta\, , \quad \partial_y \beta = c_1\, e^{-\lambda x}\, , \quad \partial_z\beta = 0\, ,
\end{equation*}

\noindent
whose general solution is given by $\beta = (c_1 y + 	c_2 ) e^{-\lambda x}$ for $c_2\in \mathbb{R}$. Hence, if $(M,g)$ admits a stationary Killing flow, which in particular implies that $(M,g)$ is static, then $(M,g)$ is isomorphic to the following model (assuming $c_1\neq 0$):
\begin{equation*}
(M,g) = (\mathbb{R}\times (-\frac{c_2}{c_1},\infty) \times \mathbb{R}^2, - (c_1 y + c_2 )^{2} e^{-2\lambda x} \dd t^2 + \dd x^2 + e^{-2 \lambda x} (\dd y^2 + \dd z^2))\, .
\end{equation*}

\noindent
In the limit $c_1\to 0$ we obtain the universal cover of AdS$_4$, which is well-known to admit real Killing spinors \cite{Alonso-Alberca:2002wsh,GibbonsRuback}. For $c_1\neq 0$ we obtain a one-parameter deformation of the universal cover of AdS$_4$ which preserves its diffeomorphism type.
\end{example}

\begin{example}
Consider a globally hyperbolic four-dimensional space-time $(M,g)$ of the form:
\begin{equation*}
(M,g) = (\cI\times \Sigma, -\beta^2_t \dd t\otimes \dd t + h )\, ,
\end{equation*}
	
\noindent
where $\cI \subset \mathbb{R}$ is an interval and $(\Sigma,h)$ is a complete and simply connected Riemannian three-manifold. We have $\Theta_t = 0$ and therefore restricting the Killing spinorial flow equations to the Cauchy surface $\Sigma$ we obtain, as in the previous example, the following equations: 
\begin{eqnarray*}
\dd e_u +  \lambda\, e_l \wedge e_u = 0\, , \quad \dd e_n + \lambda\, e_l\wedge e_n = 0\, , \quad [ e_l ] = 0 \in H^1(\Sigma,\mathbb{R})\, ,
\end{eqnarray*}  
	
\noindent
where $\fre = (e_u,e_l,e_n) := \fre^0$. Hence, as in the previous example there exist global coordinates $(x,y,z)$ identifying $\Sigma = \mathbb{R}^3$ in which the metric $h_{\fre}$ takes the form:
\begin{equation*}
h_{\fre} = \dd x \otimes \dd x + e^{-2\lambda x} (\dd y \otimes \dd y + \dd z\otimes \dd z)\, .
\end{equation*}
 
\noindent 
By Theorem \ref{thm:Killingspinorflow}, the Killing spinorial flow equations reduce to:
\begin{eqnarray}
\label{eq:globhyperspinorflowIts}
& \partial_t e^t_u + \dd\beta_t(e^t_u) e^t_u =   \dd\beta_t + \lambda \beta_t e_l^t\, , \quad \partial_t e^t_l  =  -(\dd\beta_t(e^t_l) + \lambda \beta_t) e_u^t\, ,\\
& \partial_t e^t_n  =  - \dd\beta_t(e^t_n) e_u^t\, , \quad   \lambda \, \partial_t e^t_l = - \dd (\dd \beta_t (e_u^t))\, , \label{eq:globhyperspinorflowIIts} \\
& \dd e^t_u = \lambda e^t_u  \wedge e^t_l\, , \quad \dd e^t_l= 0  \, , \quad \dd e^t_n = \lambda e^t_n\wedge e^t_l\, , \quad [e^t_l]=0\in H^1(\Sigma,\mathbb{R})\, .\label{eq:globhyperspinorflowIIIts} 
\end{eqnarray}
	
\noindent 
Assume for simplicity $\partial_t e^t_l = 0$. Under these assumptions, the previous equations further simplify to:
\begin{eqnarray*}
& \partial_t e^t_u  = b_t e_n^t\, , \quad \partial_t e^t_n  =  - b_t  e_u^t\, ,\\
& \dd e^t_u = \lambda e^t_u  \wedge e^t_l\, , \quad \dd e^t_l= 0  \, , \quad \dd e^t_n = \lambda e^t_n\wedge e^t_l\, , 
\end{eqnarray*}
	
\noindent 
where we have written $\dd \beta_t = a_t e_u^t-\lambda \beta_t e_l^t+ b_t e_n^t$ in terms of a family $\left\{ a_t \right\}_{t\in \cI}$ of constants and a family $\left\{ b_t \right\}_{t\in \cI}$ of functions on $\mathbb{R}^3$. Define a family of complex one-forms $\left\{ z^t \right\}_{t\in \cI}$ as $z^t = e_u^t + i e_n^t$ for every $t\in \cI$. Then, the previous equations are equivalent to:
\begin{equation*}
\partial_t z^t = - i b_t\, z^t\, , \qquad \dd z^t = \lambda z^t \wedge e_l\, ,
\end{equation*}
	
\noindent
where we have set $e_l = \dd x$ and $\dd \beta_t = a_t e_u^t-\lambda \beta_t e_l^t+ b_t e_n^t$. Note that the latter condition defines a differential equation for $\beta_t$. Given that $h_{\fre^t}$ does not depend on $t$, there exists a smooth family of maps $\left\{ \Gamma_t\colon \mathbb{R}^3\to \mathrm{SO}(2)\subset \mathbb{C}^{\ast}\right\}_{t\in\cI} $ such that $z^t = \Gamma_t z^0$. Since $\mathbb{R}^3$ is simply connected we have $\Gamma_t = e^{i\theta_t}$ in terms of a family of real functions $\left\{ \theta_t\right\}$ on $\mathbb{R}^3$. Plugging $z^t = \Gamma_t z^0$ into the previous equations, we obtain:
\begin{equation*}
\partial_t \theta_t = - b_t\, , \qquad \dd\theta_t \wedge z^0 = 0\, ,
\end{equation*}
	
\noindent
where $\fre = (e_u, e_l, e_n) = (e^{-\lambda x} \dd y, \dd x , e^{-\lambda x} \dd z)$. Equation $\dd\theta_t \wedge z^0 = 0$ is equivalent to $\dd\theta_t = 0$ whence $\dd b_t = 0$ and $\left\{ b_t\right\}_{t\in \cI}$ is a family of constants on $\mathbb{R}^3$. On the other hand, equation $\dd \beta_t = a_t e_u^t-\lambda \beta_t e_l + b_t e_n^t$ is equivalent to:
\begin{equation*}
\partial_x \beta_t = - \lambda \beta_t\, , \quad \partial_y \beta_t = (a_t \cos(\theta_t) + b_t \sin(\theta_t)) e^{-\lambda x}\, , \quad \partial_z \beta_t = (b_t \cos(\theta_t) - a_t \sin(\theta_t)) e^{-\lambda x}\, ,
\end{equation*}
	
\noindent
where $\theta_t = - \int_0^t b_{\tau}\dd\tau$ and we have used equation $z^t = e^{i\theta_t} z^0$ in order to express $\fre^t$ in terms of $\fre$ as follows:
\begin{equation}
\label{eq:etexample}
e^t_u = \cos(\theta_t) e^{-\lambda x} \dd y - \sin(\theta_t) e^{-\lambda x} \dd z\, , \quad e^t_l = \dd x\, , \quad e^t_n = \sin(\theta_t) e^{-\lambda x} \dd y + \cos(\theta_t) e^{-\lambda x} \dd z\, .
\end{equation}
	
\noindent
The general solution to the previous equations can be found to be:
\begin{equation}
\label{eq:betatejemplo}
\beta_t e^{\lambda x} = (a_t \cos(\int_0^t b_{\tau}\dd\tau) + b_t \sin(\int_0^t b_{\tau}\dd\tau)) y + (b_t \cos(\int_0^t b_{\tau}\dd\tau) - a_t \sin(\int_0^t b_{\tau}\dd\tau)) z  + c_t(x)\, ,
\end{equation}
	
\noindent
where $\left\{c_t\right\}_{t\in \cI}$ is a family of constants. Therefore we conclude that, given the families $\left\{c_t\right\}_{t\in \cI}$, $\left\{a_t\right\}_{t\in \cI}$ and $\left\{b_t\right\}_{t\in \cI}$ as introduced above, equations \eqref{eq:betatejemplo} and \eqref{eq:etexample} define a real Killing spinorial flow $\left\{ \beta_t , \fre^t\right\}_{t\in\cI}$ on $\mathbb{R}^3$, perhaps after restricting $\cI$ to a sub-interval guaranteeing that $\beta^2_t > 0$. 
\end{example}


\subsection{The constraint equations}
\label{eq:constraint}


The real Killing spinorial flow equations pose an evolution problem whose associated constraint equations happen to be equivalent to the constraint equations of the evolution problem of a real Killing spinor on a globally hyperbolic Lorentzian four-manifold. By Theorem \ref{thm:Killingspinorflow}, the variables of the Killing spinorial flow equations correspond to families $\left\{ \beta_t , \fre^t\right\}_{t\in \cI}$ of functions and coframes on $\Sigma$ whose evolution is prescribed by equations \eqref{eq:globhyperspinorflowI} and \eqref{eq:globhyperspinorflowII}. Equations \eqref{eq:globhyperspinorflowIII} and \eqref{eq:globhyperspinorflowIV} determine the constraint equations that the allowed initial data of the evolution problem must satisfy on a given Cauchy surface. If we take $\Sigma := \Sigma_0$ as Cauchy hypersurface and set:
\begin{equation*}
\fre := \fre^0\, , \quad \Theta := \Theta^0\, ,
\end{equation*}

\noindent
then the restriction of equations \eqref{eq:globhyperspinorflowIII} and \eqref{eq:globhyperspinorflowIV} to $\Sigma$ is given by:
\begin{eqnarray}
\label{eq:exderKI}
& \dd e_u=\Theta(e_u)\wedge e_u + \lambda e_u \wedge e_l\, , \quad \dd e_l=\Theta(e_l) \wedge e_u  \, , \quad \dd e_n   = \Theta(e_n) \wedge e_u + \lambda e_n\wedge e_l \, ,\\
& [\Theta(e_u) + \lambda e_l ] = 0 \in H^1(\Sigma,\mathbb{R})\, .
\label{eq:exderKII}
\end{eqnarray}

\noindent
We will consider equations \eqref{eq:exderKI} and \eqref{eq:exderKII} as the constraint equations of Killing spinorial flow, to which we will refer as the \emph{Killing Cauchy differential system} if $\lambda \neq 0$, a condition that we will assume in the following. 

\begin{remark}
The Killing Cauchy differential system contains a \emph{cohomological} condition, namely $[\Theta(e_u) + \lambda e_l] = 0$, which restricts the admissible discrete quotients that a given solution admits, since an exact one-form on $\Sigma$ does not necessarily descend to an exact one-form on a given discrete quotient of $\Sigma$. 
\end{remark}

\begin{definition}
\label{def:KCC}
A \emph{Killing Cauchy pair} $(\fre,\Theta)$ is a solution of the Killing Cauchy differential system with $\lambda \neq 0$. A \emph{parallel Cauchy pair} $(\fre,\Theta)$ is a Killing Cauchy pair with $\lambda = 0$.
\end{definition}

\noindent
We denote by $\Conf(\Sigma)$ the configuration space of the Killing Cauchy differential system, that is, its space of variables $(\fre,\Theta)$, whereas we denote by $\Sol(\Sigma)$ the space of Killing Cauchy pairs. Note that the function $\beta$ does not occur in the Killing Cauchy differential system, exactly as it happens in the initial value problem posed by the Ricci-flat or Einstein condition on a Lorentzian metric \cite{Yvonne,ChoquetBruhatBook}. As a direct consequence of the previous discussion together with Theorem \ref{thm:Killingspinorflow} we obtain the following result.
 
\begin{proposition}
\label{prop:spinoronrflat}
A globally hyperbolic four-manifold $(M,g)$ admits a Killing spinor $\varepsilon\in\Gamma(\mathrm{S}_g)$ only if it admits a Cauchy surface $\Sigma \hookrightarrow M$ with second fundamental form $\Theta\in \Gamma(T^{\ast}\Sigma\odot T^{\ast}\Sigma)$ such that $\Sigma$ admits a coframe satisfying equations \eqref{eq:exderKI} and \eqref{eq:exderKII}. Furthermore, $(M,g)$ is Einstein, that is, $g$ satisfies:
\begin{equation*}
\mathrm{Ric}^g = \Lambda g\, , \qquad \Lambda \in  \mathbb{R}\, ,
\end{equation*}

\noindent
only if $h$ satisfies in addition the following equations:
\begin{equation}
\label{eq:EinsteinonSigma}
\mathrm{R}_h = \vert\Theta\vert^2_h - \mathrm{Tr}_h(\Theta)^2 + 2 \Lambda\, ,\qquad  \dd \mathrm{Tr}_h(\Theta) = \mathrm{div}_h(\Theta)\, ,
\end{equation}

\noindent
on $\Sigma$.
\end{proposition}

\noindent
The system of equations \eqref{eq:EinsteinonSigma} corresponds to the well-known \emph{constraint equations} of the Einstein equations in the presence of a \emph{cosmological constant} \cite{Yvonne,ChoquetBruhatBook}. The first equation in \eqref{eq:EinsteinonSigma} is usually called the \emph{Hamiltonian constraint} whereas the second equation in \eqref{eq:EinsteinonSigma} is usually called the \emph{momentum constraint}. Given a pair $(\fre,\Theta) \in \Conf(\Sigma)$, we denote by $h_{\fre}$ the Riemannian metric on $\Sigma$ defined by:
\begin{equation*}
h_{\fre} = e_{u}\otimes e_u + e_{l}\otimes e_l  + e_{n}\otimes e_n\, ,
\end{equation*}
 
\noindent
where $\fre = (e_u,e_l,e_n)$. We say that $(\fre,\Theta)$ is complete if $h_{\fre}$ is a complete Riemannian metric on $\Sigma$. Denote by $\mathrm{Met}(\Sigma) \times \Gamma(T^{\ast}\Sigma\odot T^{\ast}\Sigma)$ the set of pairs consisting of Riemannian metrics and symmetric two-tensors on $\Sigma$. We obtain a canonical map:
\begin{equation*}
\Psi\colon \Conf(\Sigma) \to \mathrm{Met}(\Sigma) \times \Gamma(T^{\ast}M\odot T^{\ast}M)\, , \qquad (\fre,\Theta) \mapsto (h_{\fre},\Theta)\, .
\end{equation*}

\noindent
The set $\mathrm{Met}(\Sigma) \times \Gamma(T^{\ast}M\odot T^{\ast}M)$ is in fact the configuration space of the constraint equations \eqref{eq:EinsteinonSigma}. Therefore, the map $\Psi$ provides a natural link between the initial value problem posed by a real Killing spinor and the initial value problem posed by the Einstein equations. In particular, it allows introducing a natural notion of \emph{admissible} initial data to both evolution problems. 

\begin{definition}
A real Killing Cauchy pair $(\fre,\Theta)$ is \emph{constrained Einstein} if $(h_{\fre},\Theta)$ satisfies the momentum and Hamiltonian constraints \eqref{eq:EinsteinonSigma}.
\end{definition}

\noindent
The Ricci and scalar curvature of $(\fre,\Theta)$ can be computed explicitly, which shows that Killing Cauchy pairs $(\fre,\Theta)$ are not constrained Einstein in general. In the following the symbol $\nabla^{\fre}$ will denote the Levi-Civita connection associated to $h_{\fre}$. 

\begin{proposition}
Let $(\fre, \Theta)$ be a Killing Cauchy pair. The Ricci curvature $\mathrm{Ric}^{\fre}$ of $h_{\fre}$ reads:
\begin{equation}
\mathrm{Ric}^{\fre}=\Theta \circ \Theta-\mathrm{Tr}_\fre (\Theta) \Theta+ (\dd \mathrm{Tr}_\fre(\Theta)-\mathrm{div}_\fre (\Theta) ) \otimes e_u+\nabla_{e_u}^\fre \Theta-(\nabla^\fre \Theta) (e_u)-2 \lambda^2 h\, .
\end{equation}

\noindent
In particular, the Ricci scalar $\mathrm{Scal}^{\fre}$ of $h_{\fre}$ is given by:
\begin{equation}
\mathrm{Scal}^{\fre}=\vert \Theta \vert_\fre^2-\mathrm{Tr}_\fre(\Theta)^2+2  (\dd \mathrm{Tr}_\fre(\Theta)(e_u^\sharp)-\mathrm{div}_\fre (\Theta)(e_u^\sharp) )-6\lambda^2\, .
\end{equation}
\end{proposition}

\begin{proof}
The Killing Cauchy differential system \eqref{eq:exderKI} and \eqref{eq:exderKII} implies that the orthonormal frame $\fre$ satisfies:
\begin{equation}
\nabla^\fre e_a^\sharp=\Theta(e_a)\otimes e_u^\sharp-\delta_{ua} \Theta^\sharp+\lambda e_a \otimes e_l^\sharp-\delta_{al} \lambda \, h^\sharp \, , \quad a=u,l,n\, .
\label{eq:cdframe}
\end{equation}

\noindent
Using the previous equation we can directly compute Riemann curvature of $h_{\fre}$, obtaining:
\begin{eqnarray*}
& \mathrm{R}^\fre(e_a,e_b)e_c =\Theta(e_a,e_c)(\Theta(e_b))^\sharp-\Theta(e_b,e_c) (\Theta(e_a))^\sharp+(\nabla_{e_a} \Theta)(e_b,e_c) e_u^\sharp -(\nabla_{e_b} \Theta)(e_a,e_c) e_u^\sharp\\ & +\delta_{u c} \left [( (\nabla_{e_b}\Theta)(e_a) )^\sharp- (\nabla_{e_a}\Theta)(e_b) )^\sharp\right ]+\lambda^2(h(e_a,e_c)e_b^\sharp-h(e_b,e_c) e_a^\sharp)\, ,
\end{eqnarray*}
where $a, b, c = u,l,n$. The computation of the Ricci curvature and the subsequent scalar curvature follows now directly. 
\end{proof}
 
\begin{lemma}
\label{lemma:momc}
Let $(\fre=(e_u,e_l,e_n), \Theta)$ be a Killing Cauchy pair. Then:
\begin{equation*}
\mathrm{div}_\fre (\Theta) \wedge e_u=\dd \mathrm{Tr}_\fre(\Theta)\wedge e_u\, .
\end{equation*}
\end{lemma}

\begin{proof}
The statement is equivalent to proving that the following equations hold:
\begin{equation*}
 \mathrm{div}_\fre (\Theta) (e_l)=\dd \mathrm{Tr}_\fre(\Theta)(e_l) \, , \quad \mathrm{div}_\fre (\Theta) (e_n)=\dd \mathrm{Tr}_\fre(\Theta)(e_n)
\end{equation*}

\noindent
for a Killing Cauchy pair. Using equation \eqref{eq:cdframe} we compute:
\begin{equation*}
\mathrm{div}_\fre (\Theta) (e_n)-\dd \mathrm{Tr}_\fre(\Theta)(e_n)=e_n(\tuu)+e_n(\tll)-e_u(\tun)-e_l(\tln)+\tun \mathrm{Tr}_\fre (\Theta)+3 \lambda \tln\, ,
\end{equation*}

\noindent
where $\Theta=\Theta_{ab} e^a \otimes e^b$. Now we observe that the integrability condition $\dd^2 e_l=0$ yields the constraint:
\begin{equation*}
e_l(\tln)-e_n(\tll)-\tll \tun-2 \lambda \tln+\tln \tul=0\, ,
\end{equation*}

\noindent
whence we are left with:
\begin{equation*}
\mathrm{div}_\fre (\Theta) (e_n)-\dd \mathrm{Tr}_\fre(\Theta)(e_n)=e_n(\tuu)-e_u(\tun)-\tll \tun +\tln \tul+\lambda \tln+\tun \mathrm{Tr}_\fre (\Theta)\, ,
\end{equation*}

\noindent
which implies $\mathrm{div}_\fre (\Theta) (e_n)=\dd\mathrm{Tr}_\fre(\Theta)(e_n)$ after using $\dd (\Theta(e_u)+\lambda e_l)(e_n,e_u) = 0$. Equation $\mathrm{div}_\fre (\Theta) (e_l)=\dd\mathrm{Tr}_\fre(\Theta)(e_l)$ is proven similarly and we conclude.
\end{proof}
 
\begin{corollary}
\label{cor:hammomeq}
For a Killing Cauchy pair $(\fre,\Theta)$, the vanishing of the Hamiltonian constraint is equivalent to the vanishing of the momentum constraint, that is, the first equation in \eqref{eq:EinsteinonSigma} implies the second in \eqref{eq:EinsteinonSigma} and vice versa. 
\end{corollary}

\begin{corollary}
A pair $(\fre, \Theta) \in \mathrm{Conf}(M)$ is a constrained Einstein Killing Cauchy pair if and only if:
\begin{eqnarray*}
& \dd e_u= (\Theta(e_u)-\lambda e_l) \wedge e_u\, ,  \quad \dd e_l=\Theta(e_l) \wedge e_u \, ,  \quad \dd e_n=\Theta(e_n)\wedge e_u +\lambda e_n \wedge e_l\, , \\
& \left[ \Theta(e_u)+\lambda e_l \right]=0 \, ,  \quad   \dd \mathrm{Tr}_h(\Theta) = \mathrm{div}_h(\Theta)\, ,
\end{eqnarray*}

\noindent
where $h_\fre$ is the Riemannian metric associated to $(\fre, \Theta)$. 
\end{corollary}


\subsection{Compatibility with the Einstein condition}


Let $\cK(\Sigma)$  be the set of real Killing spinorial flows on $\Sigma$, that is, the set of families $\left\{\beta_t,\fre^t\right\}_{t\in\cI}$ satisfying the real Killing spinorial flow equations \eqref{eq:globhyperspinorflowI} and \eqref{eq:globhyperspinorflowII}. We have a canonical map:
\begin{equation*}
\Phi \colon \cK(\Sigma) \to \mathrm{Lor}_{\circ}(M)\, , \qquad \left\{\beta_t,\fre^t\right\}_{t\in\cI} \mapsto  g= -\beta_t^2 \dd t^2 + h_{\fre^t}\, ,
\end{equation*}

\noindent
from $\cK(\Sigma)$ to the set $\mathrm{Lor}_{\circ}(M)$ of globally hyperbolic Lorentzian metrics on $M=\cI\times \Sigma$. Given a real Killing spinorial flow $\left\{\beta_t,\fre^t\right\}_{t\in\cI}$, there exists a smooth family functions $\left\{ \frf_t\right\}_{t\in\cI}$ such that:
\begin{equation*}
\dd\frf_t = - \Theta_t(e^t_u)-\lambda e_l^t\, , \qquad \partial_t \frf_t  = \dd \beta_t (e^t_u)\, , 
\end{equation*}

\noindent
which is unique modulo the addition of a real constant. Using this family of functions, we obtain a natural map:
\begin{equation*}
\Xi \colon \cK(\Sigma) \to \cB(M)\, , \qquad \left\{\beta_t,\fre^t\right\}_{t\in\cI} \mapsto  ([u =  e^{\frf_t} (\beta_t \dd t + e^t_u)], [l=e^t_l])\, ,
\end{equation*}

\noindent
from the set of real Killing spinorial flows on $\Sigma$ to the set $\cB(M)$ of parabolic pairs on $M$ with respect to the globally hyperbolic metric canonically defined by the given real Killing spinorial flow in which we identify null one-forms that differ by a multiplicative constant. The previous maps essentially provide an inverse construction to the splitting and reduction which was implemented at the beginning of Section \ref{sec:GloballyHyperbolic}. This allows us to relate properties of a real Killing spinorial flow to properties of the corresponding globally hyperbolic four-dimensional Lorentzian metric. For further reference, we introduce the \emph{Hamiltonian function} $H$ of a real Killing spinorial flow $\left\{\beta_t,\fre^t\right\}_{t\in\cI}$ as follows:
\begin{equation*}
H \colon M = \mathbb{R}\times \Sigma \to \mathbb{R}\, , \qquad (t,p) \mapsto (\mathrm{R}_{h_t}- \vert \Theta_t \vert_{h_t}^2+\mathrm{Tr}_{h_t} (\Theta_t)^2+6 \lambda^2)\vert_p\, ,
\end{equation*}

\noindent
where $h_t := h_{\fre^t}$ denotes the three-dimensional metric restricted to the Cauchy surface $\Sigma_t$ and $\mathrm{R}_{h_{\fre^t}}$ its scalar curvature. 

\begin{proposition}
\label{prop:ricci}
Let $\left\{\beta_t,\fre^t\right\}_{t\in\cI}$ be a real Killing spinorial flow on $\Sigma$. The Ricci curvature of $g = \Phi(\left\{\beta_t,\fre^t\right\})$ reads:
\begin{equation}
\label{eq:ricci4d}
 \mathrm{Ric}^g =-3 \lambda^2 g+ \frac{1}{2} H e^{- 2 \frf_t}  u\otimes u\, ,
\end{equation}

\noindent
where $\Xi(\left\{\beta_t,\fre^t\right\}) = ([u],[l])$ and $[u] = [e^{\mathfrak{f}_t}(\beta_t\dd t + e^t_u)]$.
\end{proposition}

\begin{remark}
Recall that two luminous one-forms satisfy $u_1, u_2 \in [u]$ if and only if $u_1 = c u_2$ for a non-zero real constant $c\in \mathbb{R}$.
\end{remark}

\begin{proof}
Let $\left\{\beta_t,\fre^t\right\}_{t\in\cI}$ be a real Killing spinorial flow on $\Sigma$ and let $g = \Phi(\left\{\beta_t,\fre^t\right\}) = -\beta^2_t \dd t \otimes \dd t + h_t$ be its associated globally hyperbolic metric on $M = \mathbb{R}\times \Sigma$. The pair $\left\{\beta_t,\fre^t\right\}_{t\in\cI}$ defines a global orthonormal coframe $(e_0,e_1,e_2,e_3)$ on $(M,g)$ given by:
\begin{equation*}
e_0\vert_{(t,p)} : = \beta_t\vert_{p}  \dd t \, , \qquad e_1\vert_{(t,p)}  := e^t_u\vert_{p} \, , \qquad e_2\vert_{(t,p)}  := e^t_l\vert_{p} \, , \qquad e_3\vert_{(t,p)}  := e^t_n\vert_{p} \, .
\end{equation*}

\noindent
The fact that $\left\{\beta_t,\fre^t\right\}_{t\in\cI}$ is a real Killing spinorial flow implies that the exterior derivatives of the coframe $(e_0,e_1,e_2,e_3)$ on $M$ are given by:
\begin{eqnarray*}
& \dd e_0 = \dd\log(\beta_t) \wedge e_0 \, , \quad \\
& \dd e_a = (\dd\log(\beta_t)(e_a)\, e_1 +  \Theta_t(e_a) - \delta_{a1} (\dd \log(\beta_t)+ \lambda e_2)) \wedge e_0 + (\Theta_t (e_a)-\lambda\delta_{a2}e_0 )\wedge e_1+ \lambda e_a \wedge e_2 \, ,
\end{eqnarray*}

\noindent
where $a = 1,2,3$. Interpreting the previous expression as the first Cartan structure equations for $\nabla^{g}$ with respect to the orthonormal coframe $(e_0,e_1,e_2,e_3)$ and using repeatedly Equations \eqref{eq:globhyperspinorflowI} and \eqref{eq:globhyperspinorflowII}, a somewhat tedious calculation yields equation \eqref{eq:ricci4d} and we conclude.
\end{proof}

\begin{theorem}
\label{thm:preservingconstraints}
The real Killing spinorial flow preserves the momentum and Hamiltonian constraints.
\end{theorem}

\begin{proof}
The proof is identical to that of \cite[Theorem 3.9]{Murcia:2021psf} for the case $\lambda =0$ but we include it for completeness. Let $\left\{\beta_t,\fre^t\right\}_{t\in\cI}$ be a get:
\begin{equation*}
\dd(H e^{-2\frf_t}) (u^{\sharp}) = 0\, ,
\end{equation*}

\noindent
which can be written in the following equivalent way:
\begin{equation*}
\mathfrak{D}(H) = f\, H\, ,
\end{equation*}

\noindent
where $\mathfrak{D}$ is a first-order symmetric hyperbolic differential operator and $f$ is a function which is completely specified by $\left\{\beta_t,\fre^t\right\}_{t\in\cI}$. Endowed with such $\mathfrak{D}$ and $f$ associated to $\left\{\beta_t,\fre^t\right\}_{t\in\cI}$, let us consider now the initial value problem given by:
\begin{equation*}
\mathfrak{D}(F) = f\, F\, , \qquad F\vert_{\Sigma} = 0\, ,
\end{equation*}

\noindent
for an arbitrary function $F$ on $M$. By the existence and uniqueness theorem for this particular type of equations, see \cite[Theorem 19]{CortesKronckeLouis} and references therein, every solution needs to be zero on a neighborhood of $\Sigma$. Noting that $H$ is a solution of this equation, it has to vanish on a neighborhood of $\Sigma$. Hence there exists a subinterval $\cI^{\prime} = (a,b) \subseteq \cI$ which contains zero and such that $H\vert_t = 0$ for every $t\in \cI^{\prime}$. Corollary \ref{cor:hammomeq} then ensures that the momentum constraint is also satisfied for every $t\in \cI^{\prime}$ and, therefore, the real Killing spinorial flow preserves the Hamiltonian and momentum constraints on $\cI^{\prime}$. In case $\cI^{\prime} = \cI$  we conclude, so let us assume that $\cI^{\prime} = (a,b) \subset \cI$ is the proper maximal subinterval of $\cI$ for which the result holds. Since the real Killing spinorial flow $\{\beta_t, \fre^t\}_{t \in \cI}$ is well-defined in $\cI$, then both $f$ and $H$ need to be well-defined on $\cI\times \Sigma$. Consequently, by point-wise continuity on $\Sigma$, $H\vert_{b}=0$ and we may apply the previous argument to the initial value problem \emph{starting} at $b \in \cI$. Therefore, there exists an $\varepsilon > 0$ for which the result holds on $(a, b+\varepsilon)$, in contradiction with the initial assumption of $(a,b)$ being maximal. Hence, $\cI^{\prime} = \cI$ and we conclude.
\end{proof}

\noindent
Call a triple $(\Sigma,h,\Theta)$ an \emph{initial Einstein} data if $(h,\Theta)$ satisfies the Hamiltonian and momentum constraints. The previous theorem can be applied to prove a partial \emph{initial data} characterization of real Killing spinors on Einstein Lorentzian four-manifolds. 

\begin{corollary}
An initial Einstein data $(\Sigma,h,\Theta)$ admits an Einstein Lorentzian development carrying a real Killing spinor only if there exists a global orthonormal coframe $\fre$ on $\Sigma$ such that $(\fre,\Theta)$ is a real Killing Cauchy pair.
\end{corollary}

\noindent
In order to promote the previous corollary to an \emph{if and only if} result we would need to prove that the Killing spinor equation is well posed on a globally hyperbolic Lorentzian manifold. To the best of our knowledge this result is yet not available in the smooth category, see \cite{LeistnerLischewski,Lischewski:2015cya} for a proof of well-posedness of the $\lambda =0$ case. On the other hand, reference \cite{Conti} contains a proof in the real analytic case, see \cite[Theorem 5.4.]{Conti}. Additionally, we obtain the following corollary.

\begin{corollary}
A globally hyperbolic Lorentzian four-manifold $(M,g)$ admitting a real Killing spinor is Einstein if and only if there exists a Cauchy hypersurface $\Sigma\subset M$ whose Hamiltonian constraint vanishes.
\end{corollary}


\section{Left-invariant real Killing spinorial flows and initial data}


Let $\G$ be a three-dimensional Lie group. We say that a Killing spinorial flow $\left\{ \beta_t ,\fre^t \right\}_{t\in \cI}$ defined on $\G$ is \emph{left-invariant} if both $\beta_t$ and $\fre^t$ are left-invariant for every $t\in \mathbb{R}$. In particular, $h_{\fre^t}$ is a left-invariant Riemannian metric and $\beta_t$ is constant for every $t\in \cI$. 


\subsection{Left-invariant Killing Cauchy pairs}


Let $\left\{ \beta_t ,\fre^t \right\}_{t\in \cI}$ be a left-invariant Killing spinorial flow on a connected and simply connected Lie group with corresponding Killing Cauchy pair $(\fre,\Theta)$. Then both $\fre$ and $\Theta$ are left-invariant so we can write:
\begin{equation}
\Theta=\Theta_{ab} e^a \otimes e^b \, , \qquad \Theta_{ab} \in \mathbb{R}\, , \qquad a,b=u,l,n\, .
\end{equation} 

\noindent
Using this expression for $\Theta$, the Killing Cauchy differential system reduces to:
\begin{eqnarray*}
&\dd e_u=(\tul e_l +\tun e_n) \wedge  e_u+\lambda e_u \wedge\, e_l \, , \quad \dd e_l= (\tll e_l+\tln e_n) \wedge e_u\, ,\\& \dd e_n=(\tln e_l+\tnn e_n) \wedge e_u+\lambda e_n \wedge e_l\,, \quad \dd(\tuu e_u+\tul e_l+\tun e_n+\lambda e_l)=0\, .
\end{eqnarray*}

\noindent
By applying the exterior derivative to the previous equations, we find the following \emph{integrability conditions}:
\begin{equation}
\label{eq:intecond}
\tun=\tln=0\, , \quad \tuu \tul+\lambda \tll+\tll \tul=\lambda \tuu\, , \quad \tnn \tul+\lambda \tll=\lambda \tnn\, .
\end{equation}

\noindent
Using these conditions, we can classify the isomorphism type of the simply connected Lie groups which admit left invariant Killing Cauchy pairs.

\begin{proposition}
Let $(\fre, \Theta)$ be a left-invariant Killing Cauchy pair on a Lie group $\G$. Then, one and only one of the following holds:
\begin{itemize}[leftmargin=*]
\item $\G \simeq \mathrm{E}(1,1)$ and $\tul=2\lambda$, $\tll=-\tnn$, $\tuu=3 \tnn$.
\item $\G \simeq \tau_2 \oplus \mathbb{R}$ and $\tul=\lambda$, $\tll=0$.
\item $\G \simeq \tau_{3,\mu}$ and $\tul\neq \lambda, 2\lambda$, $\tll=\tnn\left (1-\frac{\tul}{\lambda} \right)$, $\tuu=\left ( 1+\frac{\tul}{\lambda} \right) \tnn$. More concretely, $\mu=\left(1-\frac{\tul}{\lambda} \right)^\sigma $, where $\sigma=1$ if $ \left \vert 1-\frac{\tul}{\lambda}  \right \vert \leq 1$ and $\sigma=-1$ otherwise.
\end{itemize}
\end{proposition}

\begin{remark}
$\mathrm{E}(1,1)$ denotes the group of rigid motions of two-dimensional Minkowski space, $\tau_2 \oplus \mathbb{R}$ stands\footnote{The notation $\tau_2 \oplus \mathbb{R}$ and $\tau_{3,\mu}$ is adopted from \cite{Gorbatsevich}.} for the connected and simply connected Lie group whose Lie algebra is the direct sum of the unique non-abelian 2-dimensional Lie algebra $\tau_2$ with $\mathbb{R}$ and  $\tau_{3,\mu}$ denotes the connected and simply connected Lie group whose Lie algebra is the semi-direct sum $\mathbb{R}\loplus_\varphi \mathbb{R}^2$, with $\varphi(1)=\begin{pmatrix}
1 & 0\\
0 & \mu\
\end{pmatrix}$, where $-1 < \mu \leq 1$ and $\mu \neq 0$.
\end{remark}
\begin{proof}
If $\tul=2\lambda$, equation \eqref{eq:intecond} immediately implies that $\tll=-\tnn$ and $\tuu=3 \tnn$. Hence the exterior derivatives of the left-invariant coframe reduce to:
\begin{equation*}
\dd e_u=\lambda e_l \wedge e_u \, , \quad \dd e_l=-\tnn e_l \wedge e_u\, , \quad \dd e_n=\tnn e_n \wedge e_u+\lambda e_n \wedge e_l\, .
\end{equation*}

\noindent
Defining $e_1=e_n$, $e_2=\lambda e_u-\tnn e_l$ and $e_3=\tnn e_u+\lambda e_l$, we obtain:
\begin{equation*}
 \dd e_1=e_1 \wedge e_3\, , \quad \dd e_2=-e_2 \wedge e_3\, , \quad \dd e_3=0  \,.
\end{equation*}

\noindent
which, after some additional algebraic manipulations, implies in turn that $\G\simeq \mathrm{E}(1,1)$ \cite{Freibert,Milnor}. 

Assume now $\tul\neq 2 \lambda$. We proceed by finding all solutions ot the following integrability equations in \eqref{eq:intecond}:
\begin{equation*}
\tuu \tul+\lambda \tll+\tll\tul=\lambda \tuu\, , \quad \tnn \tul+\lambda \tll=\lambda \tnn\, .
\end{equation*}

\noindent
By the second equation it is clear that $\tll=\tnn\left (1-\frac{\tul}{\lambda} \right)$. Substituting this expression into the first equation, we are left with the following condition:
\begin{equation*}
(\lambda-\tul) \left ( 1+\frac{\tul}{\lambda} \right)\tnn+\tuu (\tul-\lambda)=0\,.
\end{equation*}

\noindent
Clearly $\tul=\lambda$ is a solution, which in turn implies $\tll=0$. In this case, the exterior derivative of the coframe $\fre$ is given by:
\begin{equation*}
\dd e_u=0\, , \quad \dd e_l=0\, , \quad \dd e_n=\tnn e_n \wedge e_u+\lambda e_n \wedge e_l\, ,
\end{equation*}

\noindent
whence $\G\simeq \tau_2 \oplus \mathbb{R}$ \cite{Freibert,Milnor}. If $\tul \neq \lambda, 2\lambda$, then we find:
\begin{equation*}
\tuu=\left ( 1+\frac{\tul}{\lambda} \right) \tnn\, ,
\end{equation*}

\noindent
The exterior derivative of the coframe $\fre$ reads:
\begin{equation*}
\dd e_u=(\tul-\lambda) e_l \wedge e_u \, , \quad \dd e_l= \tnn \left ( 1-\frac{\tul}{\lambda} \right ) e_l \wedge e_u \, , \quad \dd e_n=\tnn e_n \wedge e_u+\lambda e_n \wedge e_l\, .
\end{equation*}

\noindent
Defining $e_1=e_n$, $e_2=(\tul-\lambda) e_u+\tnn \left ( 1-\frac{\tul}{\lambda} \right ) e_l$ and $e_3=  \tnn \left ( 1-\frac{\tul}{\lambda} \right ) e_u-(\tul-\lambda) e_l$, we find:
\begin{equation*}
\dd e_1=\frac{1}{1-\frac{\tul}{\lambda}}e_1\wedge e_3\, , \quad \dd e_2=e_2 \wedge e_3 \, , \quad \dd e_3=0\, .
\end{equation*}

\noindent
Taking into account that $ 1-\frac{\tul}{\lambda} \neq -1,0$, we conclude that $\mathrm{G}$ is isomorphic to $\tau_{3,\mu}$, with $\mu=\left(1-\frac{\tul}{\lambda} \right)^\sigma $, where $\sigma=1$ if $ \left \vert 1-\frac{\tul}{\lambda}  \right \vert \leq 1$ and $\sigma=-1$ otherwise.
\end{proof}

\noindent
We consider now the constrained Einstein condition. 

\begin{proposition}
A left-invariant Killing Cauchy pair $(\fre, \Theta)$ on $\G$ is a constrained Einstein if and only if:
\begin{equation*}
\mathrm{div}_\fre (\Theta)(e_u)=\tll^2+2\tln^2 +\tnn^2+2 \tul^2+2\tun^2-\tll \tuu-\tnn\tuu-3 \tul \lambda=0\, .
\end{equation*}
\end{proposition}

\begin{proof}
By Lemma \ref{lemma:momc} and Corollary \ref{cor:hammomeq} it is enough to impose that:
\begin{equation*}
\mathrm{div}_\fre (\Theta)(e_u)=0\, .
\end{equation*}

\noindent
A direct computation, by virtue of \eqref{eq:cdframe}, gives the desired result. 
\end{proof}

\noindent
The previous discussion implies the following classification result.

\begin{theorem}
\label{theo:consks}
A simply connected Lie group $\G$ admits left-invariant Killing Cauchy pairs (respectively constrained Einstein Killing Cauchy pairs\footnote{The third column of the Table should be understood as the additional condition that a Killing Cauchy pair has to satisfy in order to be constrained Einstein.}) if and only if $\G$ is isomorphic to one of the Lie groups listed in the Table below. If that is the case, a left-invariant shape operator $\Theta$ belongs to a Cauchy pair $(\fre,\Theta)$ for certain left-invariant coframe $\fre$ if and only if $\Theta$ is of the form listed below when written in terms of $\fre=(e_u,e_l,e_n)$:
\begin{center}
\begin{tabular}{|  p{1cm}| p{9cm} | p{3.5cm} |}
\hline
$\mathrm{G}$ & \emph{Killing Cauchy pair} & \emph{Constrained Einstein}  \\ \hline
\multirow{2}*{$ \mathrm{E}(1,1)$} & \multirow{2}*{$\Theta=3\tnn e_u \otimes e_u+ 2 \lambda e_u \odot e_l-\tnn e_l \otimes e_l+\tnn e_n \otimes e_n$} & \multirow{2}*{\emph{Not allowed}}  \\ & &  \\ \hline \multirow{2}*{$ \tau_2 \oplus \mathbb{R}$} & \multirow{2}*{$\Theta=\Theta_{uu} e_u \otimes e_u+  \lambda e_u \odot e_l+\tnn e_n \otimes e_n$} & \multirow{2}*{$\tnn(\tnn-\tuu)=\lambda^2$} \\ & &  \\   \hline
\multirow{4}*{$ \tau_{3,\mu}$} & \multirow{3}*{$\Theta=\Theta_{uu} e_u \otimes e_u+ \tul  e_u \odot e_l+\tll e_l \otimes e_l+\tnn e_n \otimes e_n$} & \multirow{4}*{$\tul(2 \tul-3 \lambda )=0$}\\ & &   \\ & & \\ & $\tul \neq \lambda, 2\lambda\, , $ $  \lambda \tll= \left (\lambda-\tul \right )\tnn \, , $ $ \lambda \tuu=\left (\lambda+\tul\right )\tnn $ & \\\hline
\end{tabular}
\end{center}
 
\noindent
If $\G \simeq \tau_{3,\mu}$ we have $\mu=\left(1-\frac{\tul}{\lambda} \right)^\sigma $, where $\sigma=1$ if $ \left \vert 1-\frac{\tul}{\lambda}  \right \vert \leq 1$ and $\sigma=-1$ otherwise.
\end{theorem}


\subsection{The integrability conditions of a left-invariant Killing spinorial flow}


Any square matrix $\mathcal{P}\in \mathrm{Mat}(3,\mathbb{R})$ acts on $\left\{ \fre^t\right\}_{t\in \cI}$ in a natural way as follows: 
\begin{equation*}
\mathcal{P}(\fre^t) := 
\begin{pmatrix}
\sum_b \mathcal{P}_{ub} e^t_b  \\
\sum_b \mathcal{P}_{lb} e^t_b  \\
\sum_b \mathcal{P}_{nb} e^t_b
\end{pmatrix}
\end{equation*}

\noindent
where we are labeling the entries $\mathcal{P}_{ab}$ of $\mathcal{P}$ by the indices $a,b = u, l, n$. Define $\cA \in \mathrm{Mat}(3,\mathbb{R})$ as the matrix such that $\cA (e_u^t)=e_l^t$, $\cA (e_l^t)=-e_u^t$ and $\cA(e_n^t)=0$. As a direct consequence of Theorem \ref{thm:Killingspinorflow} we have the following result.
\begin{proposition}
A simply connected three-dimensional Lie group  $\G$ admits a left-invariant Killing spinorial flow if and only if there exists a smooth family of non-zero constants $\left\{ \beta_t\right\}_{t\in \cI}$ and a family $\left\{ \fre^t \right\}_{t\in \cI}$ of left-invariant coframes on $\G$ satisfying the following differential system:
\begin{align}
\label{eq:leftinv1}\partial_t \fre^t  +  \beta_t\Theta_t(\fre^t) &= \lambda \beta_t \mathcal{A} (\fre^t)\, , \quad \dd \fre^t  = \Theta_t(\fre^t) \wedge e^t_u+\lambda \fre^t \wedge e_l^t\, , \\ \label{eq:leftinv2} \quad \partial_t(\Theta_t(e^t_u )+\lambda e_l^t) &= 0\, , \quad \dd\Theta_t(e^t_u )+\lambda \dd e_l^t = 0\, ,
\end{align}
	
\noindent
to which we will refer as the left-invariant (real) Killing spinorial flow equations.
\end{proposition}

\noindent
We will refer to solutions $\left\{ \beta_t, \fre^t \right\}_{t\in \cI}$ of the left-invariant real Killing spinorial flow equations as \emph{left-invariant real Killing spinorial flows}. Given a real Killing spinorial flow $\left\{ \beta_t, \fre^t \right\}_{t\in \cI}$, we write:
\begin{equation*}
\Theta^t = \sum_{a,b} \Theta^t_{ab} e^t_a\otimes e^t_n\, , \qquad a, b = u, l , n\, ,
\end{equation*}

\noindent
in terms of uniquely defined functions $(\Theta^t_{ab})$ on $\cI$.

\begin{lemma}
\label{lemma:intecondli}
Let $\{\beta_t, \fre^t\}_{t \in \mathcal{I}}$ be a left-invariant real Killing spinorial flow. The following equations hold:
\begin{equation*}
\begin{split}
\partial_t \tuu^t=\beta_t (\lambda^2 +2\lambda \tul^t+ (\tuu^t)^2 + (\tul^t )^2)\,,& \quad \partial_t \tul^t=0 \, , \quad \tun^t=\tln^t=0\, , \\ \partial_t \tll^t=\beta_t(\tll^t \tuu^t+\lambda^2-(\tul^t)^2)\,,   &
 \quad \partial_t \tnn^t=\beta_t(\tnn^t \tuu^t+\lambda^2+\lambda \tul^t)\, , \\
\lambda \tll^t=\tnn^t (\lambda-\tul^t) \, , & \quad \lambda \tuu^t(\tul^t-\lambda)+ \tnn^t (\lambda^2-(\tul^t)^2 )=0\,.
\end{split}
\end{equation*}

\noindent
In particular, $\tul^t=\tul$  for some $\tul  \in \mathbb{R}$.
\end{lemma}

\begin{proof}
A direct computation shows that equation $\partial_t(\Theta_t(e^t_u )+\lambda e_l^t) =0 $ is equivalent to:
\begin{equation*}
\partial_t \Theta^t_{ub} = \beta_t \Theta^t_{ua} \Theta^t_{ab}+\lambda \Theta_{lb}^t-\lambda \beta_t \Theta_{ua}^y \cA_{ab}+\lambda^2 \beta_t \delta_{ub}\, ,
\end{equation*}

\noindent
where we remind the reader that $\cA_{lu}=-1$ and $\cA_{ul}=1$. On the other hand, equations $\dd \Theta_t(e^t_u) +\lambda \dd e_l^t= 0$ and $\dd(\Theta_t(\fre^t) \wedge e_u^t+\lambda \fre^t \wedge e_l^t)=0$ are equivalent to:
\begin{equation*}
\tun^t=\tln^t=0\, , \quad \tuu^t \tul^t+\lambda \tll^t+\tll^t \tul^t=\lambda \tuu^t\, , \quad \tnn^t \tul^t+\lambda \tll^t=\lambda \tnn^t\, .
\end{equation*}
\noindent
The previous equations can be combined into the following equivalent conditions:
\begin{eqnarray*}
& \partial_t \tuu^t=\beta_t (\lambda^2 +2\lambda \tul^t+ (\tuu^t)^2 + (\tul^t )^2)\, , \quad \partial_t \tul^t=0\, ,\\
& \lambda \tll^t=\tnn^t (\lambda-\tul^t) \, , \quad \lambda \tuu^t(\tul^t-\lambda)+ \tnn^t (\lambda^2-(\tul^t)^2 )=0\, , \quad \tun^t=\tln^t=0\,.
\end{eqnarray*}
\noindent
which recover five of the equations in the statement.  We take now the exterior derivative of the first equation in \eqref{eq:leftinv1} and combine the result with the second equation in \eqref{eq:leftinv1}:
\begin{eqnarray*}
& 0 = \dd (\partial_t e^t_a  +  \beta_t\Theta_t(e^t_a) -\lambda \beta_t  \cA(e_a^t)) =  \partial_t (\Theta^t_{ab} e^t_b\wedge e^t_u+ \lambda e_a^t \wedge e_u^t) + \beta_t \Theta^t_{ab} \dd e^t_b- \lambda  \beta_t \cA_{ab} \dd e_b^t  \\
&=  (\partial_t \Theta^t_{ab} \delta_{uc} - \beta_t \Theta^t_{ab} \Theta^t_{uc}+\lambda \beta_t ((\Theta^t_{ae} \cA_{eb}-\cA_{ae} \Theta^t_{eb}) \delta_{uc} +\Theta^t_{lb} \delta_{ac}+\Theta^t_{ab} \cA_{uc})  + \lambda^2 \beta_t \cA_{lc} \delta_{ab}) e^t_b\wedge e^t_c  
\end{eqnarray*}

\noindent
Expanding the previous equation we obtain the remaining two equations in the statement and hence we conclude.
\end{proof}

\begin{remark}
We will refer to the equations of Lemma \ref{lemma:intecondli} as the \emph{integrability conditions} of the left-invariant real Killing spinorial flow.
\end{remark}

\noindent
The following observation is crucial in order to \emph{decouple} the left-invariant real Killing spinorial flow equations.

\begin{lemma}
\label{lemma:cKTheta}
A pair $\left\{ \beta_t, \fre^t \right\}_{t\in \cI}$ is a left-invariant Killing spinorial flow if and only if there exists  a family of left-invariant two-tensors $\left\{ \cK_t \right\}_{t\in \cI}$ such that the following equations are satisfied:
\begin{eqnarray*}
\partial_t \fre^t  +  \beta_t\cK_t(\fre^t) = \lambda \beta_t \cA(\fre^t)\, ,  & \quad \dd \fre^t  = \cK_t(\fre^t) \wedge e^t_u+\lambda \fre^t \wedge e_l^t\, ,\\ \partial_t(\cK_t(e^t_u )+\lambda e_l^t) = 0\, , & \quad \dd(\cK_t(e^t_u)+\lambda e_l^t) = 0\, .
\end{eqnarray*}
\end{lemma}

\begin{proof}
The \emph{only if} direction follows immediately from the definition of left-invariant real Killing Cauchy pair by taking $\{\cK_t\}_{t \in \cI}=\{\Theta_t\}_{t \in \cI}$. For the \emph{if} direction we simply compute:
\begin{equation*}
\Theta_t = - \frac{1}{2\beta_t}\partial_t h_{\fre^t} = - \frac{1}{2\beta_t} (( \partial_t e^t_a)\otimes e^t_a + e^t_a \otimes ( \partial_t e^t_a)) = \cK_t\, ,
\end{equation*}

\noindent
where we have take into account that $\cA$ is an antisymmetric endomorphism. Hence equations \eqref{eq:leftinv1} are satisfied and $\left\{ \beta_t, \fre^t \right\}_{t\in \cI}$ is a left-invariant real Killing spinorial flow.
\end{proof}

\noindent
The previous Lemma allows us to promote the components of $\left\{\Theta_t \right\}_{t\in \cI}$ with respect to the basis $\left\{\fre^t \right\}_{t\in \cI}$ to be independent variables of the left-invariant Killing spinorial flow equations \eqref{eq:leftinv1}. Following this interpretation, the variables of left-invariant real Killing spinorial flow equations turn out to be triples\footnote{We shall denote $\left\{\beta_t \right\}_{t\in \cI}$ equivalently as $\left\{\beta^t \right\}_{t\in \cI}$ to unify notation along the section.} $\left\{\beta^t,\fre^t,\Theta^t_{ab} \right\}_{t\in \cI}$, where we denote by $\left\{\Theta^t_{ab} \right\}_{t\in \cI}$ a family of symmetric matrices. On the other hand, the integrability conditions of Lemma \ref{lemma:intecondli} are to be interpreted as a system of equations for a pair $\left\{\beta^t,\Theta^t_{ab} \right\}_{t\in \cI}$. More concretely, the first equation in \eqref{eq:leftinv1} is linear in the variable $\fre^t$ and may be rewritten as follows. For any family of coframes $\left\{\fre^t\right\}_{t\in\cI}$, define $\fre = \fre^0$ and consider the unique smooth path:
\begin{equation*}
\U^t \colon \cI \to \Gl_{+}(3,\mathbb{R})\, , \qquad t\mapsto \U^t\, ,
\end{equation*}

\noindent
such that $\fre^t = \U^t(\fre)$, where  $\Gl_{+}(3,\mathbb{R})$ stands for the identity component in the general linear group $\Gl(3,\mathbb{R})$. In particular:
\begin{equation*}
\fre^t_a = \sum_b \U^t_{ab} \fre_b\, , \qquad a, b = u, l ,n\, ,
\end{equation*}

\noindent
where $\U^t_{ab} \in C^{\infty}(\G)$ are the components of $\U^t$. Plugging $\fre^t = \U^t(\fre)$ in the first equation in \eqref{eq:leftinv1} we obtain the following equivalent equation:
\begin{equation}
\label{eq:ptu}
\partial_t \U^t_{ac} + \beta^t \Theta^t_{ab} \U^t_{bc} = \lambda \beta_t \mathcal{A}_{ab} \U^t_{bc}\, ,  \quad a, b, c = u, l, n\, ,
\end{equation}

\noindent
with initial condition $\U^0 = \mathrm{Id}$. A necessary condition for a solution $\left\{\beta^t,\Theta^t_{ab} \right\}_{t\in \cI}$ of the integrability conditions to come from a left-invariant Killing spinor pair is the existence of a left-invariant coframe $\fre$ on $\Sigma$ such that $(\fre,\Theta)$ is a Killing Cauchy pair, where $\Theta = \Theta^0_{ab} e_a\otimes e_b$. We define the set $\mathbb{I}(\Sigma)$ of \emph{admissible} solutions to the aforementioned integrability equations of Lemma \ref{lemma:intecondli} as the set of pairs $(\left\{\beta^t,\Theta^t_{ab} \right\}_{t \in \cI} , \fre )$ such that $\left\{\beta^t,\Theta^t_{ab} \right\}_{t\in \cI}$ solves the integrability equations of Lemma \ref{lemma:intecondli} and $(\fre,\Theta)$ is a left-invariant Killing Cauchy pair. 

\begin{proposition}
\label{prop:bijectionsolutions}
There exists a natural bijection $\Psi\colon \mathbb{I}(\Sigma) \to \cP(\Sigma)$ mapping every pair:
\begin{equation*}
(\left\{\beta^t,\Theta^t_{ab} \right\}_{t\in\cI} , \fre )\in \mathbb{I}(\Sigma)\, ,
\end{equation*}

\noindent
to  $\left\{ \beta^t,\fre^t = \U^t(\fre)\right\}_{t\in\cI}\in \cP(\Sigma)$, where $\left\{\U^t\right\}_{t\in\cI}$ is the unique solution of \eqref{eq:ptu} with initial condition $\U^0 = \mathrm{Id}$. 
\end{proposition}

\begin{remark}
\label{remark:inverse}
The inverse of $\Psi$  maps every left-invariant Killing spinorial flow $\left\{ \beta^t,\fre^t\right\}_{t\in\cI}$ to $(\left\{\beta^t,\Theta^t_{ab} \right\} , \fre )$, where $\Theta^t_{ab}$ denotes the components of the shape operator associated to $\left\{ \beta^t,\fre^t\right\}_{t\in\cI}$ in the basis $\left\{\fre^t\right\}_{t\in\cI}$ and $\fre = \fre^0$.
\end{remark}

\begin{proof}
Let $(\left\{\beta^t,\Theta^t_{ab} \right\} , \fre )\in \mathbb{I}(\Sigma)$ and let  $\left\{\U^t\right\}_{t\in\cI}$ be the solution of \eqref{eq:ptu} with the initial condition $\U^0 = \mathrm{Id}$, which is guaranteed to exist and to be unique on $\cI$ by standard ODE theory \cite[Theorem 5.2]{CoddingtonLevinson}. We must prove that $\left\{ \beta^t,\fre^t = \U^t(\fre)\right\}_{t\in\cI}$ is a left-invariant Killing spinorial flow. Since $\left\{\U^t\right\}_{t\in\cI}$  fulfills \eqref{eq:ptu} for the given $\left\{\beta^t,\Theta^t_{ab} \right\}$, we have that $\Theta^t = \Theta^t_{ab} e^t_a\otimes e^t_b$ is the shape operator corresponding to $\left\{ \beta^t,\fre^t \right\}_{t\in\cI}$, so the first equation in \eqref{eq:leftinv1} is satisfied. On the other hand, the two equations in \eqref{eq:leftinv2} follow by the integrability conditions satisfied by  $\left\{\beta^t,\Theta^t_{ab} \right\}$. Regarding the second equation in \eqref{eq:leftinv1}, we note that the integrability conditions imply the equation $\dd (\Theta_t(\fre^t) \wedge e_u^t+\lambda \fre^t \wedge e_l^t)=0$ and thus:
\begin{equation}
\label{eq:constraintintegrada}
\dd \fre^t= \Theta^t(\fre^t) \wedge e_u^t+\lambda \fre^t \wedge e_l^t+ \mathfrak{w}^t\, ,
\end{equation}

\noindent
where $\{ \mathfrak{w}^t\}_{t \in \cI}$ is a family of closed two-forms on $\Sigma$. Taking the time derivative of the previous equations, substituting the exterior derivative of the first equation into \eqref{eq:leftinv1} and making use again of the integrability conditions, we observe that $\mathfrak{w}^t$ satisfies the following differential equation:
\begin{equation}
\label{eq:frwode}
\partial_t \mathfrak{w}^t_a = -\beta_t \Theta^t_{ad} \mathfrak{w}^t_d+\lambda \beta_t \cA_{ad} \mathfrak{w}_d^t\, ,
\end{equation}

\noindent
with initial condition $\mathfrak{w}^0 = \mathfrak{w}$. Evaluating equation \eqref{eq:constraintintegrada} at $t=0$, we get:
\begin{equation*}
\dd \fre= \Theta(\fre) \wedge e_u + \lambda \fre \wedge e_l+  \mathfrak{w}\, ,
\end{equation*}

\noindent
Since $(\fre,\Theta)$ is by assumption a left-invariant Killing Cauchy pair, the previous equation holds if and only if $\mathfrak{w} =0$, so $\mathfrak{w}^t = 0$ by uniqueness of solutions of the linear differential equation \eqref{eq:frwode}. Hence the second equation in \eqref{eq:leftinv1} is satisfied and $\Psi$ is well-defined. The fact that $\Psi$ is a bijection as well follows directly by Remark \ref{remark:inverse} and we conclude. 
\end{proof}
 
\begin{corollary}
\label{cor:solutionsiff}
A pair $\left\{\beta^t,\fre^t \right\}_{t\in \cI}$ is a Killing spinorial flow if and only if $(\left\{\beta^t,\Theta^t_{ab} \right\} , \fre )$ is an admissible solution to the integrability equations.
\end{corollary}
 
\noindent
Consequently, solving the left-invariant Killing spinorial flow is equivalent to solving the integrability conditions with initial condition $\Theta$ being part of a left-invariant real Killing Cauchy pair $(\fre,\Theta)$. We note that $\left\{\beta^t \right\}_{t\in \cI}$ is of no relevance locally since it can be eliminated through a reparametrization of time after possibly deforming $\cI$. However, owing to the long time existence of the flow as well as for applications for the construction of four-dimensional Lorentzian metrics, it proves to be convenient to keep track of $\cI$, so we shall maintain $\left\{\beta^t \right\}_{t\in \cI}$ in the equations. 

In the following we proceed to classify left-invariant Killing spinorial flows. For that, we distinguish between the three possible isomorphism types for $\G$.


\subsection{Case $\mathrm{E}(1,1)$}


Let us assume $\G = \mathrm{E}(1,1)$. Then by Theorem \ref{theo:consks} we have that $\tul=2\lambda$, which upon use of the integrability conditions in Lemma \ref{lemma:intecondli}, implies that $\tll^t=-\tnn^t$ and $\tuu^t= 3 \tnn^t$. Substituting these results into the remaining integrability conditions, they reduce to:
\begin{equation*}
\partial_t \tnn^t=3 \beta_t (\lambda^2+(\tnn^t)^2)\, ,
\end{equation*}

\noindent
whose solution is:
\begin{equation}
\tnn^t=\lambda \tan y_t \,, \quad y_t=\arctan\left (\frac{\tnn}{\lambda} \right) +3\lambda \cB_t  \, ,
\end{equation}

\noindent
where we have defined $\cB_t=\int_0^t \beta_\tau \dd \tau$.  

\begin{proposition}
Let $\{\beta_t, \fre^t\}_{t \in \cI}$ be a left-invariant Killing spinorial flow on $\G =\mathrm{E}(1,1)$. Then:
\begin{align}
e_u^t=\U^t_{uu} e_u+\U^t_{ul} e_l \, , & \quad e_l^t=\U^t_{lu} e_u +\U^t_{ll} e_l\,, \quad
e_n =\left( \frac{\cos y_t }{\cos y_0} \right)^{1/3} e_n\,,
\end{align}
where:
\begin{align*}
\U^t_{uu}=\frac{\tnn}{\lambda} \mathfrak{R}(y_t)+\frac{\cos^{2/3} y_t}{ \cos^{2/3} y_0}\,, &\quad \U^t_{ul}= \mathfrak{R}(y_t)\,, \\ \U^t_{ll}=-3\, \U^t_{ul} \tan y_t  - \frac{\partial_t \U^t_{ul}}{\lambda \beta_t}\,,  & \quad \U^t_{lu}=-3\,  \U^t_{uu}\tan y_t  - \frac{\partial_t \U^t_{uu}}{\lambda \beta_t}\,.
\end{align*}
with\footnote{Note that $ {}_2F_1 \left (\frac{1}{2},  \frac{1}{6}; \frac{7}{6};0 \right )=\frac{ \sqrt{\pi} \, \Gamma \left ( 7/6\right) }{\Gamma \left (2/3 \right)}$.}:
\begin{align*}
\mathfrak{R}(x)&=\left (  {}_2F_1 \left (\frac{1}{2},  \frac{1}{6}; \frac{7}{6};\cos^2 x \right ) \cos x- \frac{ \sqrt{\pi} \, \Gamma \left ( 7/6\right) }{\Gamma \left (2/3 \right)} \cos^{2/3} x   \right) \mathrm{sign}(x)-\mathfrak{R}_0  \cos^{2/3} x\,, \\
 \mathfrak{R}_0 &=\mathrm{sign}(y_0) \left ( {}_2F_1 \left (\frac{1}{2},  \frac{1}{6}; \frac{7}{6};\cos^2 y_0 \right ) \cos^{1/3} y_0-\frac{ \sqrt{\pi} \, \Gamma \left ( 7/6\right) }{\Gamma \left (2/3 \right)}\right ) \,.
\end{align*}
The solution is defined in the connected open interval containing $t=0$ in which $\vert y_t \vert < \frac{\pi}{2}$. 
\end{proposition}

\begin{proof}
If $\G = \mathrm{E}(1,1)$, using equation \eqref{eq:ptu} we find:
\begin{equation*}
\partial_t \U^t_{uc} +\lambda \beta_t   \U^t_{lc}+3  \beta_t \tnn^t \U^t_{uc}=0\,, \quad \partial_t \U^t_{lc}+3 \lambda \beta_t \U_{uc}^t- \beta_t \tnn^t \U_{lc}^t=0\, , \quad \partial_t \U^t_{nc}+ \beta_t \tnn^t \U^t_{nc}=0\,.
\end{equation*}

\noindent
The first two equations with $c=n$ together with the third equation have a unique solution given by:
\begin{equation*}
\U^t_{un}=\U^t_{ln}=\U^t_{nu}=\U^t_{nl}=0\, , \quad \U^t_{nn}=\left( \frac{\cos y_t }{\cos y_0} \right)^{1/3}\, .
\end{equation*}

\noindent
The equations that remain to be solved are:
\begin{equation*}
\partial_t \U^t_{uj} +\lambda \beta_t   \U^t_{lj}+3  \beta_t \tnn^t \U^t_{uj}=0\,, \quad \partial_t \U^t_{lj}+3 \lambda \U_{uj}^t- \beta_t \tnn^t \U_{lj}^t=0\, ,  \quad j=u,l\, .
\end{equation*}

\noindent
By differentiating the other equations and carrying out the appropriate substitutions, we find that the previous equations are equivalent to:
\begin{eqnarray*}
&-\lambda \beta_t \U^t_{lj}-  3  \beta_t \tnn^t \U^t_{uj} + \partial_t \U^t_{uj} =0\, , \\ 
&\partial^2_t \U^t_{uj}- 3 \U^t_{uj} \beta_t (\beta_t(\lambda^2+ (\tnn^t)^2)-\partial_t \tnn^t )+2 \beta_t \tnn^t \partial_t \U^t_{uj}-\frac{\partial_t \beta_t}{\beta_t} \partial_t \U_{uj}^t =0\, .
\end{eqnarray*}

\noindent
After the change of variables $\tau= \cB_t$, one can rewrite the latter second-order differential equation as:
\begin{equation}
\frac{\dd^2 \U^t_{uj}}{\dd \tau^2}+2 \lambda \frac{\dd}{\dd \tau} ( \U_{uj}^t \tan y_t)=0\, .
\end{equation}

\noindent
The general solution to the previous differential equation is:
\begin{equation*}
\U^t_{uj}=\left ( \frac{\cos y_t}{\cos y_0 } \right)^{2/3} \left ( c_{1j}+ c_{2j} \int_0^t \left ( \frac{\cos y_0}{\cos y_\sigma } \right)^{2/3} \beta_\sigma \dd \sigma  \right) \, ,
\end{equation*}

\noindent
where we have expressed the solutions in terms of original time variable and $c_{1j},c_{2j} \in \mathbb{R}$. The integral above can be expressed in terms of hypergeometric functions ${}_2 F_1 (a,b;c;z)$. Requiring that $\U^{t=0}_{uj}=\delta_{uj}$ and demanding continuity of $\U^t_{uj}$ and its first derivative, we learn that they must take the form shown at the statement of the Proposition. Fixing the rest of the constants to ensure that $\U^t_{lj}=\delta_{lj}$, we conclude. The solution is defined in the connected open interval containing $t=0$ in which $\vert y_t \vert < \frac{\pi}{2}$.
\end{proof}


\subsection{Case $\tau_2 \oplus \mathbb{R}$}


If $\G =\tau_2 \oplus \mathbb{R}$, then by Theorem \ref{theo:consks} we have that $\tul=\lambda$. Plugging this equation into Lemma \ref{lemma:intecondli}, we infer that $\tll^t=0$ and the integrability conditions which remain to be solved read:
\begin{equation*}
\partial_t \tuu^t=\beta_t ((\tuu^t)^2+4\lambda^2)\, , \quad \partial_t \tnn^t=\beta_t (\tuu^t \tnn^t+2\lambda^2)\,.
\end{equation*}

\noindent
The solution to the previous ordinary differential equation is:
\begin{equation*}
\tuu^t= 2\lambda \tan x_t \,, \quad \tnn^t= \frac{(2\tnn-\tuu) }{\sqrt{4+\frac{\tuu^2}{\lambda^2}}}\sec x_t  + \lambda \tan x_t  \, , \quad x_t= \arctan \left ( \frac{\tuu}{2\lambda}\right)+2\lambda \cB_t\,.
\end{equation*}

\noindent
Plugging these results into \eqref{eq:ptu}, we can find the explicit form of left-invariant Killing spinorial flows on $\tau_2 \oplus \mathbb{R}$.
\begin{proposition}
Let $\{\beta_t,\fre^t\}_{t \in \cI}$ a left-invariant Killing spinorial flow on $\tau_2 \oplus \mathbb{R}$. Then:
\begin{align}
e_u^t&=\sqrt{1+\frac{\tuu^2}{4\lambda^2}}\, \cos x_t e_u \, , \quad e_l^t=\left (\frac{\tuu}{2\lambda}- \sqrt{1+\frac{\tuu^2}{4\lambda^2}}\sin x_t\right ) e_u+e_l\, , \\ e_n^t&=\left ( \left ( 1+\frac{\tuu^2}{4\lambda^2}\right)  \cos^2 x_t\right)^{1/4} \left \vert \frac{\cos\left( \frac{x_t+x_0}{2}\right)+\sin\left( \frac{x_t-x_0}{2}\right) }{\cos\left( \frac{x_t+x_0}{2}\right)-\sin\left( \frac{x_t-x_0}{2}\right)} \right\vert^\gamma e_n\, , 
\end{align}
where $\gamma=\dfrac{(\tuu-2\tnn) }{2\lambda \sqrt{4+\frac{\tuu^2}{\lambda^2}}}$. 
\end{proposition}

\begin{proof}
By the discussion above, equation \eqref{eq:ptu} reduces to:
\begin{equation*}
\partial_t \U^t_{uc}+\beta_t \tuu^t \U^t_{uc}=0\, , \quad \partial_t \U^t_{lc}+2\lambda\beta_t \U^t_{uc} =0\, , \quad \partial_t \U^t_{nc}+\beta_t \tnn^t \U^t_{nc}=0\,.
\end{equation*}
The solution to the previous system of ODEs is:
\begin{align*}
\U^t_{ul}&=\U^t_{un}=\U^t_{ln}=\U^t_{nu}=\U^t_{nl}=0\, , \quad \U^t_{ll}=1\, , \\
\U^t_{uu}&=\sqrt{1+\frac{\tuu^2}{4\lambda^2}}\, \cos x_t\, , \quad \U^t_{lu}=\frac{\tuu}{2\lambda}- \sqrt{1+\frac{\tuu^2}{4\lambda^2}}\sin x_t \\
\U^t_{nn}&=\left ( \left ( 1+\frac{\tuu^2}{4\lambda^2}\right)  \cos^2 x_t\right)^{1/4} \left \vert \frac{\cos\left( \frac{x_t+x_0}{2}\right)+\sin\left( \frac{x_t-x_0}{2}\right) }{\cos\left( \frac{x_t+x_0}{2}\right)-\sin\left( \frac{x_t-x_0}{2}\right)} \right\vert^\gamma\,,
\end{align*}
where $\gamma=\dfrac{(\tuu-2\tnn) }{2\lambda \sqrt{4+\frac{\tuu^2}{\lambda^2}}}$ and hence we conclude.
\end{proof}

\begin{remark}
Let $t_-<0$ denote the largest value for which $x_t=-\frac{\pi}{2}$ and let $t_+>0$ denote the smallest value for which $x_t=\frac{\pi}{2}$ (if $t_-, t_+$ or both do not exist, we take by convention $t_{\pm}=\pm \infty$). Then it can be seen that the maximal interval of definition of the flow is $\mathcal{I}=(t_-,t_+)$.
\end{remark}


\subsection{Case $\tau_{3,\mu}$}


Let us assume $\G = \tau_{3,\mu}$. In this case, $\tul \neq \lambda, 2 \lambda$ and therefore the integrability conditions of Lemma \ref{lemma:intecondli} imply that:
\begin{equation*}
\tuu^t=\frac{\tul+\lambda}{\lambda} \tnn^t \, , \quad \tll^t=-\frac{\tul-\lambda}{\lambda} \tnn^t\, .
\end{equation*}

\noindent
The remaining integrability conditions are tantamount to:
\begin{equation*}
\partial_t \tnn^t=2 \beta_t( \lambda^2+(\tnn^t)^2)\, ,
\end{equation*}

\noindent
whose solution is:
\begin{equation*}
\tnn^t=\lambda \tan z_t\, ,\quad z_t=\arctan \left ( \frac{\tnn}{\lambda} \right ) +(\tul+\lambda) \cB_t \, .
\end{equation*}

\begin{proposition}
Let $\{\beta_t, \fre^t\}_{t \in \cI}$ be a left-invariant Killing spinorial flow on $\tau_{3,\mu}$ with $\tul\neq -\lambda$. Then:
\begin{align}
e_u^t=\U^t_{uu} e_u +\U^t_{ul} e_l \, , & \quad e_l^t=\U^t_{lu} e_u +\U^t_{ll} e_l \,, \quad
e_n =\left( \frac{\cos z_t }{\cos z_0} \right)^\kappa e_n\,,
\end{align}
where:
\begin{align*}
\U^t_{uu}=\frac{\tnn}{\lambda} \mathfrak{R}_\kappa(z_t)+\frac{\cos^{2\kappa} z_t}{ \cos^{2 \kappa } z_0}\,, &\quad \U^t_{ul}= \mathfrak{R}_\kappa(z_t)\, ,\\ \U^t_{ll}=-\frac{(\tul+\lambda) \beta_t\,  \U^t_{ul} \tan z_t  +\partial_t \U^t_{ul}}{(\tul-\lambda) \beta_t}\,,  & \quad \U^t_{lu}=-\frac{(\tul+\lambda) \beta_t\,  \U^t_{uu} \tan z_t  +\partial_t \U^t_{uu}}{(\tul-\lambda) \beta_t}\,.
\end{align*}
with:
\begin{align*}
\mathfrak{R}_\kappa(x)&=\left (  {}_2F_1 \left (\frac{1}{2}-\kappa,  \frac{1}{2}; \frac{3}{2}-\kappa;\cos^2 x \right ) \cos x- \frac{ \sqrt{\pi} \, \Gamma \left ( 3/2-\kappa\right) }{\Gamma \left (1-\kappa \right)} \cos^{2 \kappa} x   \right) \mathrm{sign}(x)-\mathfrak{R}_{\kappa,0}  \cos^{2 \kappa} x\,, \\
 \mathfrak{R}_{\kappa,0} &=\mathrm{sign}(z_0) \left ( {}_2F_1 \left (\frac{1}{2}-\kappa,  \frac{1}{2}; \frac{3}{2}-\kappa;\cos^2 z_0 \right ) \cos^{1-2\kappa} z_0-\frac{ \sqrt{\pi} \, \Gamma \left ( 3/2-\kappa\right) }{\Gamma \left (1-\kappa \right)}\right ) \,.
\end{align*}
where $\kappa=\frac{\lambda}{\tul+\lambda}$. The solution is defined in the connected open interval containing $t=0$ in which $\vert z_t \vert < \frac{\pi}{2}$. 
\end{proposition}

\begin{proof}
Completely analogous to that of the $\mathrm{E}(1,1)$ case.
\end{proof}

\begin{remark}
If $\tul=-\lambda$, the solution turns out to be quite simple:
\begin{align}
e_u^t=e_u-\frac{\lambda}{\tnn} \mathrm{Exp} \left( -2 \tnn \cB_t \right) e_l\, , \quad e_l^t=\mathrm{Exp} \left( -2 \tnn \cB_t \right) e_l \, , \quad  e_n^t=\mathrm{Exp} \left( - \tnn \cB_t \right) e_n\,.
\end{align}
The case $\tnn=0$ is obtained by taking the formal limit $\tnn \rightarrow 0$ in the previous equation.
\end{remark}

\begin{remark}
The three-dimensional Ricci tensor of the family of Riemannian metrics $\left\{h_{\fre^t}\right\}_{t\in \cI}$ associated to a left invariant real Killing spinorial flow is given by:
\begin{itemize}
\item Case $\G=\tau_2 \oplus \mathbb{R}$ and $\G=\tau_{3,\mu}$:
\begin{equation*}
\mathrm{Ric}^{h_{\fre^t}}=(\lambda^2+(\tnn^t)^2)(-h_{\fre^t}+\eta \otimes \eta)\,, \quad \eta=\frac{1}{\sqrt{\lambda^2 +(\tnn^t)^2}}(\lambda e_u^t- \tnn^t e_l^t)\,.
\end{equation*}
This defines a family of $\eta\,$-Einstein cosymplectic structures on $\G$, since $\nabla^{h_{\fre^t}} \eta=0$.
\item Case $\G=\mathrm{E}(1,1)$:
\begin{equation*}
\mathrm{Ric}^{h_{\fre^t}}=\frac{1}{2} H_t \eta \otimes \eta\,, \quad \eta=\frac{1}{\sqrt{\lambda^2 +(\tnn^t)^2}}(\tnn^t  e_u^t+\lambda e_l^t)\,,
\end{equation*}
where $H_t=-4 \lambda^2 \sec^2 y_t$. This defines a family of $\eta\,$-Einstein cosymplectic structures on $\G$, since $\dd \eta=\dd \star_{h_{\fre^t}} \eta=0$.
\end{itemize}
\end{remark}

\begin{corollary}
Let $\left\{\beta_t,\fre^t\right\}_{t\in\cI}$ be a left-invariant real Killing spinor in $(M,g)$. Then the Hamiltonian function $H_t$ reads:
\begin{itemize}[leftmargin=*]
\item If $\G=\tau_2$, $H_t=\frac{4 \lambda^2 H_{0}}{\tuu^2+4 \lambda^2}  \sec^2 x_t$.
 
\item If $\G=\mathrm{E}(1,1)$, $H_t=-4 \lambda^2 \sec^2 y_t$.

\item If $\G=\tau_{3,\mu}$, $H_t=\frac{\lambda^2 H_{0}}{\lambda^2+\tnn^2} \sec^2 z_t$.
\end{itemize}

\noindent
where $H_0$ is the Hamiltonian constraint at time $t=0$. 
\end{corollary}

\vspace{0.1cm}
\subsection*{Conflicts of interest/Competing interests statement}

\noindent
The authors have no conflicts of interest to declare that are relevant to the content of this article.

\subsection*{Data availability statement }

This manuscript has no associated data. 


\phantomsection
\bibliographystyle{JHEP}


\end{document}